\documentclass[a4paper,11pt,nosumlimits,nonamelimits]{article}
\usepackage[utf8]{inputenc}
\usepackage{amsmath,amssymb,amsfonts,amsthm}
\usepackage{array} 
\usepackage{url,hyperref,verbatim}
\usepackage{color}
\usepackage{graphicx}
\usepackage{epstopdf}
\usepackage[center]{caption}

\renewenvironment{itemize} 
	{\begin{list}
		{$\bullet$}{\setlength{\parskip}{5pt} \setlength{\topsep}{0cm}
		 \setlength{\partopsep}{0cm} \setlength{\itemsep}{3pt} \setlength{\parsep}{0cm}\item[]}}
	{\end{list}}

\theoremstyle{plain}
\newtheorem{theorem}{Theorem}
\newtheorem{lemma}{Lemma}
\newtheorem{corollary}{Corollary}

\newtheorem{proposition}{Proposition}

\theoremstyle{definition}
\newtheorem{definition}{Definition}

\theoremstyle{remark}
\newtheorem{remark}{Remark}

\newcommand{\real}{\mathbb{R}}
\newcommand{\nneg}{\mathbb{R}_+}
\newcommand{\pos}{(0,\infty)}
\newcommand{\znneg}{\mathbb{Z}_+}
\newcommand{\nat}{\mathbb{N}}

\newcommand{\la}{\langle}
\newcommand{\ra}{\rangle}

\renewcommand{\liminf}{\underline{\lim}}
\renewcommand{\limsup}{\overline{\lim}}

\newcommand{\pr}{\mathbb{P}}
\newcommand{\ex}{\mathbb{E}}
\newcommand{\ind}{\mathbb{I}}

\newcommand{\msr}{\mathbf{M}}
\newcommand{\skr}{\mathbf{D}}

\newcommand{\wto}{\stackrel{\text{w}}{\to}}

\newcommand{\cla}{\mathcal{A}}
\newcommand{\clc}{\mathcal{C}}
\newcommand{\cld}{\mathcal{D}}
\newcommand{\cli}{\mathcal{I}}
\newcommand{\cll}{\mathcal{L}}
\newcommand{\clm}{\mathcal{M}}
\newcommand{\clq}{\mathcal{Q}}
\newcommand{\clr}{\mathcal{R}}
\newcommand{\clu}{\mathcal{U}}
\newcommand{\cly}{\mathcal{Y}}
\newcommand{\clz}{\mathcal{Z}}
\newcommand{\clzinit}{\mathcal{Z}^{r, \, \text{init}}}
\newcommand{\clznew}{\mathcal{Z}^{r, \, \text{new}}}

\newcommand{\zinit}{Z^{r, \, \text{init}}}
\newcommand{\znew}{Z^{r, \, \text{new}}}

\newcommand{\clol}{\overline{\mathcal{L}}}
\newcommand{\cloy}{\overline{\mathcal{Y}}}
\newcommand{\cloz}{\overline{\mathcal{Z}}}
\newcommand{\clozinit}{\overline{\mathcal{Z}}^{r, \, \text{init}}}
\newcommand{\cloznew}{\overline{\mathcal{Z}}^{r, \, \text{new}}}

\newcommand{\oz}{\overline{Z}}
\newcommand{\oy}{\overline{Y}}
\newcommand{\ozinit}{\overline{Z}^{r, \, \text{init}}}
\newcommand{\oznew}{\overline{Z}^{r, \, \text{new}}}
\newcommand{\ove}{\overline{E}}
\newcommand{\oa}{\overline{A}}
\newcommand{\oq}{\overline{Q}}

\newcommand{\eps}{\varepsilon}
\newcommand{\Lmb}{\Lambda}
\newcommand{\lmb}{\lambda}
\newcommand{\sgm}{\sigma}
\newcommand{\Sgm}{\Sigma}
\newcommand{\Dlt}{\Delta}
\newcommand{\dlt}{\delta}

\DeclareMathOperator*{\argmax}{arg\,max}

\begin{document}

\abovedisplayskip=.6\abovedisplayskip 
\belowdisplayskip=.6\belowdisplayskip 
\abovedisplayshortskip=.6\abovedisplayshortskip 
\belowdisplayshortskip=.6\belowdisplayshortskip 

\frenchspacing 

\title{Fluid Limits for Bandwidth-Sharing Networks\\with Rate Constraints}

\author{ Maria Frolkova, Josh Reed and Bert Zwart\footnote{MF~is with
CWI, P.O. Box 94079, 1098 XG Amsterdam, The Netherlands. E-mail:
\url{M.Frolkova@cwi.nl}. JR~is with NYU Stern School of Business, 44 West 4th St., Suite 8-79, New York, NY 10012, the USA. E-mail: \url{jreed@stern.nyu.edu}. BZ~is with CWI. E-mail: \url{Bert.Zwart@cwi.nl}.
BZ~is also affiliated with EURANDOM, VU University Amsterdam, and
Georgia Institute of Technology. The research of~MF and~BZ in supported by an NWO VIDI grant.}}

\date{April 12, 2013}

\maketitle

\begin{abstract}
Bandwidth-sharing networks as introduced by Massouli\'e~\& Roberts (1998) model the dynamic interaction among an evolving population of elastic flows competing for several links. With policies based on optimization procedures, such models are of interest both from a~Queueing Theory and Operations Research perspective.

In the present paper, we focus on bandwidth-sharing networks with capacities and arrival rates of a large order of magnitude compared to transfer rates of individual flows. This regime is standard in practice. In particular, we extend previous work by Reed \& Zwart (2010) on fluid approximations for such networks: we allow interarrival times, flow sizes and patient times (i.e. abandonment times measured from the arrival epochs) to be generally distributed, rather than exponentially distributed. We also develop polynomial-time computable fixed-point approximations for stationary distributions of bandwidth-sharing networks, and suggest new techniques for deriving these types of results.

\vspace{2ex} \textit{Keywords:} bandwidth-sharing, rate constraints, impatience,  large capacity scaling, fluid limits, fixed-point approximations.

\vspace{2ex} \textit{MSC2010:} Primary 60K25, 60K30, 60F17, 60G57; Secondary 90B15, 90B22.
\end{abstract}

\section{Introduction}
Bandwidth-sharing policies as introduced by  Massouli\'{e} \& Roberts~\cite{MR98, MR99} dynamically distribute network resources among a changing population of users. Processor sharing is an example of such a~policy and assumes a single resource. Bandwidth-sharing networks are of great research and practical interest. Along with the basic application in telecommunications media, e.g. Internet congestion control, they also have recently been suggested as a tool in analyzing problems in road traffic~\cite{KW10}.

The main issues in bandwidth-sharing related research are stability conditions and performance evaluation. A variety of results regarding the first topic may be found in  De Veciana {\em et al.} \cite{VLK99, VLK01}, Bonald \& Massouli\'{e} \cite{BM01}, Mo \& Walrand \cite{MW00}, Massouli\'{e} \cite{Massoulie07}, Bramson \cite{Bramson05}, Gromoll \& Williams \cite{GW07}, and Chiang {\em et al.} \cite{CST06}. As for the second topic, for special combinations of network topologies and bandwidth-sharing policies, the network stationary distribution may be shown to be of a product form insensitive to the flow size distribution, see Bonald {\em et al.} \cite{BMPV06}. However, in general, approximation methods must be used, which is the subject matter of the present paper. Fundamental papers on fluid limit approximations for bandwidth sharing-networks are Kelly \& Williams \cite{KW04} and Gromoll \& Williams \cite{GW09}, some more results on fluid and diffusion approximations are to be found in Borst {\em et al.} \cite{EBZ07, BEZ09}, Kang {\em et al.} \cite{KKLW09} and Ye \& Yao \cite{YY08, YY}. The latter works ignore the fact that generally in practice the maximum service rate of an individual user is constrained, as has been pointed out by Roberts \cite{Roberts04}.

To the best of our knowledge, Ayesta \& Mandjes \cite{AM09} were the first to deal with fluid and diffusion approximations of bandwidth-sharing networks with rate limitations. They consider two specific settings first without rate constraints, and then they truncate the capacity constraints at the rate maxima. Reed \& Zwart \cite{RZ10} develop a different approach in the context of general bandwidth-sharing networks. They incorporate the rate constraints into the network utility maximization procedure that defines bandwidth allocations. Thus, users operating below the maximal rate are allowed to take up the bandwidth that is not used by other rate constrained users, and bandwidth allocations are Pareto optimal. Another interesting feature of this work is the scaling regime. In contrast to the papers mentioned above, which mostly focus on the large-time properties of networks with fixed-order parameters, Reed \& Zwart view networks on a fixed-time scale letting arrival rates and capacities grow large. This large capacity scaling reflects the fact that overall network capacity and individual user rate constraints may be of different orders of magnitude. For example, it is common that Internet providers set download speed limitations for individual users which are typically measured in megabits per second, while network capacities are measured in gigabits or terabits per second.

The framework of  \cite{RZ10} is rather comprehensive. In particular, it allows abandonments of flows: each flow knows how long it can stay in the system and abandons as soon as its service is finished or its patience time expires, whichever happens earlier. The present paper builds upon \cite{RZ10} by relaxing its stochastic assumptions: we assume general distribution for interarrival times and general joint distribution for the size and patience time of a flow (in particular, the flow size and patience time are allowed to be dependent), while \cite{RZ10} assumes a Markovian setting with independent arrivals, flow sizes and patience times. We study the behavior of bandwidth-sharing networks in terms of measure-valued processes that are called state descriptors and that keep track of residual flow sizes and residual patience times. The first main result of the paper is a fluid limit theorem (it generalizes the fluid limit result of~\cite{RZ10} to non-Markovian stochastic assumptions). We propose a~fluid model, or a formal deterministic approximation of the stochastic bandwidth-sharing model, and show that the scaled state descriptors are tight with all weak limit points a.s. solving the fluid model equation. We provide a sufficient condition for the fluid model to have a unique solution, which converts tightness of the scaled state descriptors into convergence to this fluid model solution. In the sense of techniques used in the proofs, this part of the paper is closely related to previous work on bandwidth-sharing \cite{GW09}, processor-sharing with impatience \cite{GRZ08}, and bandwidth-sharing in overload \cite{BEZ09, EBZ07}. The rate constraints play a crucial role in adopting these techniques. For example, the proof of convergence to fluid model solutions in \cite{GRZ08} requires an additional assumption of overload to eliminate problems at zero. However, in our case, due to the rate constraints, the network never empties, and the load conditions become irrelevant.

Our second main result, which is a new type of result for bandwidth-sharing networks, is convergence of the scaled network stationary distribution to the fixed point of the fluid model, provided the fixed point is unique. There is a similar result by Kang \& Ramanan \cite{KR10} for a call center model, but the techniques of~\cite{KR10} are different than ours. Applying the approach of Borst {\it et al.}~\cite{BEZ09}, we prove that in many cases the fixed point can be found by solving an optimization problem with a strictly concave objective function and a polyhedral constraint set, and thus is unique and computable in polynomial time. We also construct an example with multiple fixed points, which is a feature that is distinctive from earlier cited works. Besides proving new results for the particular model of bandwidth-sharing, we also suggest new ideas and believe that they can be adjusted to other models, too. In particular, we derive equations for asymptotic bounds for fluid model solutions (see~Theorem~\ref{th:stable_fixed_point}) that can be solved for a wide class of networks, and then asymptotic stability of the fixed point can be shown. Another interesting idea is that, in the stationary regime, the properties of a network depend on newly arriving flows only, since all initial flows are gone after some point (see~Lemma~\ref{lem:stat_no_atoms}). Throughout this part of the paper, we assume Poisson arrivals, since that guarantees existence of a unique stationary distribution. Poisson arrivals also imply $M/G/\infty$ bounds that are exploited heavily in the proofs.

The structure of the paper is as follows. Section~\ref{sec:stochastic_model} describes the stochastic bandwidth-sharing model, and Section~\ref{sec:fluid_model} introduces its deterministic analogue, the fluid model. Also Section~\ref{sec:fluid_model} states sufficient conditions for a fluid model solution to be unique, and for a fixed fluid model solution to be unique and asymptotically stable. Sections~\ref{sec:fluid_limits} and~\ref{sec:stat_distributions} discuss convergence of the scaled state descriptor and its stationary distribution to the fluid model and its fixed point, respectively. Sections~\ref{sec:proofs_fluid_model},~\ref{sec:proof_fluid_limits} and~\ref{sec:proof_stat_distributions} contain the proofs of the statements from Sections~\ref{sec:fluid_model},~\ref{sec:fluid_limits} and~\ref{sec:stat_distributions}. The Appendix proves auxiliary results. In the remainder of this section, we list the notation we use throughout the paper.

\paragraph{Notation} In order to introduce the notation, we use the signs $=:$ and $:=$.

The standard sets are denoted as follows: the reals $\real = (-\infty, \infty)$, the non-negative reals $\nneg = [0,\infty)$, the positive reals $(0,\infty)$, the non-negative integers $\znneg = \{ 0,1,2,\ldots \}$, and the natural numbers $\nat = \{1,2,\ldots\}$. 

The signs $\wedge$ and $\vee$ stand for minimum and maximum respectively. For $x \in \real$, $x^+ := x \vee 0$.

The signs $\liminf$ and $\limsup$ denote the lower and upper limits of a sequence of numbers.

The coordinates of a vector from a set $S^I$ are denoted by the same symbol as the vector with lower indices $1,\ldots,I$ added. If a vector has a~superscript, tilde-sign, or overlining, they remain in its coordinates. For example $\overline{x}^0 \in S^I$, $\overline{x}^0 = (\overline{x}_1^0, \ldots, \overline{x}_I^0)$. The space $\real^I$ is endowed with the supremum norm $\|x\|:= \max_{1 \leq i \leq I} |x_i|$. Vector inequalities hold coordinate-wise. The coordinate-wise product of vectors of the same dimensionality~$I$ is $x \ast y := (x_1 y_1, \ldots, x_I y_I)$.

The signs $\Rightarrow$, $\stackrel{\text{d}}{=}$ and $\leq_{\text{st}}$ stand for convergence in distribution, equality in distribution and stochastic order respectively. Recall that, for real-valued r.v.'s $X$ and $X'$, $X{\leq_{\text{st}}} X'$ if ${\mathbb{P}\{ X>x \}} \leq {\mathbb{P}\{
X' > x \}}$ for all $x \in \real$. The notation $\Pi(\lmb)$, $\lmb \in \pos$, stands for the Poisson distribution with parameter~$\lmb$.

For metric spaces $S$ and $S'$, denote by $\mathbf{C}_{S \to S'}$ the space of continuous functions $f \colon S \to S'$. By $\skr_{\nneg \to S}$ denote the space of functions $f \colon \nneg \to S$ that are right-continuous with left limits, and endow this space with the Skorokhod $J_1$-topology.

The superscript $-1$ is only used to denote the inverse of a function.

For a measure $\xi$ on $\nneg^2$ and a~$\xi$-integrable function $f \colon \nneg^2 \to \real$, define $\la f, \xi \ra := \int_{\nneg^2} f d \xi$. If $\xi = (\xi_1,\ldots, \xi_I)$ is a~vector of such measures, $ \la f, \xi \ra := ( \la f, \xi_1 \ra, \ldots, \la f, \xi_I \ra ) $. Let $\msr$ be the space of finite non-negative Borel measures on $\nneg^2$ endowed with the weak topology: $\xi^k \stackrel{\text{w}}{\to} \xi$ in $\msr$ as $k \to \infty$ if and only if $\la f, \xi^k \ra \to \la f, \xi \ra$ for all continuous bounded function $f \colon \nneg^2 \to \real$. Weak convergence of elements of $\msr$ is equivalent to convergence in the Prokhorov metric: for $\xi,\varphi \in \msr$, define
\begin{align*} 
d(\xi,\varphi) := \inf \{ & \eps \colon \xi(B) \leq \varphi(B^\eps) + \eps \text{ and } \varphi(B) \leq \xi(B^\eps) + \eps \\ 
&\text{for all non-empty closed } B \subseteq \nneg^2 \},
\end{align*}
where $B^\eps = \{ x \in \nneg^2 \colon \inf_{y \in B} \| x - y \| < \eps \}$.

For $\xi, \varphi \in \msr^I$, define
\[ d_I(\xi,\varphi) := \max_{1 \leq i \leq I} d(\xi_i, \varphi_i). \]
Equipped with the metric $d_I(\cdot, \cdot)$, the space $\msr^I$ is separable and complete.

\section{Stochastic model} \label{sec:stochastic_model}
This section contains a~detailed description of the model under consideration. In particular, it specifies the structure of the network, the policy it operates under and the stochastic dynamical assumptions. Also, a stochastic process is introduced that keeps track of the state of the network, see the state descriptor paragraph.

\paragraph{Network structure}
Consider a~network that consists of a~finite number of links labeled by $j=1,\ldots,J$. Traffic offered to the network is represented by elastic flows coming from a~finite number of classes labeled by $i=1,\ldots,I$. All class~$i$ flows are transferred through a~certain subset of links, we call it {\it route~$i$}. Transfer of a~flow starts immediately upon its arrival and is continuous with all links on the route of the flow being traversed simultaneously. Let $A$ be the $J \times I$ incidence matrix, where $A_{ji}=1$ if route~$i$ contains link~$j$ and $A_{ji}=0$ otherwise.

Suppose that at a particular time $t$ the population of the network is $z \in \znneg^I$, where $z_i$ stands for the number of flows on route~$i$. All flows on route~$i$ are transferred at the same rate~$\lmb_i(z)$ that is at most $m_i \in (0,\infty)$. If $z_i   = 0$, put $\lmb_i(z) := 0$. We refer to $\Lmb_i(z) := \lmb_i(z) z_i$ as the {\it bandwidth allocated to route~$i$}. The sum of the bandwidths allocated to the routes that contain link~$j$ is the {\it bandwidth allocated through link~$j$} and is at most $C_j \in \pos$. We call $C_j$ the {\it capacity of link~$j$}. Hence, the vectors $\lmb(z) = (\lmb_1(z),\ldots,\lmb_I(z))$ and $\Lmb(z) = (\Lmb_1(z),\ldots,\Lmb_I(z))$ must satisfy
\[
A (\lmb(z) \ast z)  = A \Lmb(z) \leq C, \quad \lmb(z) \leq m, \quad \Lmb(z) \leq m \ast z,
\]
where $C=(C_1,\ldots,C_J)$ and $m=(m_1,\ldots,m_I)$ are the vectors of link capacities and rate constraints.

\paragraph{Bandwidth-sharing policy} At each point in time, the link capacities should be distributed among the routes in such a way that the network utility is maximized. Namely, to each flow on route $i$ we assign a~utility $\clu_i(\cdot)$ that is a function of the rate allocated to that flow. Assume that the functions $\clu_i(\cdot)$ are strictly increasing and concave in $\nneg$, and twice differentiable in $\pos$ with $\lim_{x \downarrow 0} \clu'_i(x)=\infty$. We also allow $\lim_{x \downarrow 0} \clu_i(x)=-\infty$ as, for example, in the case of a logarithmic function. Then, for $z \in \nneg^I$, the vector $\lmb(z)$ of rates is the unique optimal solution to
\begin{equation} \label{eq:def_lmb}
\text{maximize} \quad \sum\nolimits_{i=1}^I z_i \, \mathcal{U}_i( \lmb_i) \quad \text{subject to} \quad A (\lmb \ast z) \leq C, \quad \lmb \leq m,
\end{equation}
where, by convention, $0 \times (-\infty) := 0$. Although the population vector has integer-valued coordinates, we assume that  $\lmb(z)$ and $\Lmb(z) := \lmb(z) \ast z$ are defined via~\eqref{eq:def_lmb} in the entire orthant $\nneg^I$ to accommodate fluid analogues of the population process later. 

The utility maximization procedure~\eqref{eq:def_lmb} implies that $\lmb_i(z) = \Lmb_i(z) = 0$ if $z_i = 0$. The assumption $\lim_{x \downarrow 0} \mathcal{U}'_i(x)=\infty$ guarantees non-idling, that is $\lmb_i(z), \Lmb_i(z) > 0$ if $z_i > 0$. Reed \& Zwart~\cite{RZ10} proved that the functions $\lmb(\cdot)$ and $\Lmb(\cdot)$ are differentiable in any direction and, in particular, locally Lipschitz continuous in the interior of $\nneg^I$. We also show continuity of $\Lmb(\cdot)$ on the boundary of $\nneg^I$ (see the Appendix).

\begin{lemma} \label{lem:Lmb_continuous}
The bandwidth allocation function~$\Lmb(\cdot)$ is continuous in $\nneg^I$.
\end{lemma}

\paragraph{Stochastic assumptions}
All stochastic primitives introduced in this paragraph are defined on a~common probability space $(\Omega, \mathcal{F}, \pr)$ with expectation operator $\ex$. 

Suppose at time zero there is an~a.s. finite number of flows in the network, we call them {\it initial flows}. A~random vector $Z^0 \in \nneg^I$ represents the initial population, and $Z_i^0$ is the number of initial flows on route~$i$. New flows arrive to the network according to a~stochastic process $E(\cdot)= (E_1(\cdot),\ldots, E_I(\cdot))$ with sample paths in the Skorokhod space $\skr_{\nneg \to \nneg^I}$. The coordinates of the arrival process are independent counting processes. Recall that a~{\it counting process} is a~non-decreasing non-negative integer-valued process starting from zero.  For $t \geq 0$, $E_i(t)$ represents the number of flows that have arrived to route $i$ during the time interval $(0,t]$. The $k$th such arrival occurs at time $U_{i k} = \inf \{ t \geq 0 \colon E_i(t) \geq k \}$, it is called {\it flow~$k$ on route~$i$}, $k \in \nat$. Simultaneous arrivals are allowed.

Flows abandon the network due to transfer completions or because they run out of patience, depending on what happens earlier for each particular flow. Flow sizes and patience times are drawn from sequences $\{ (B_{i l}^0, D_{i l}^0) \}_{l \in \nat}$, $\{ (B_{i k}, D_{i k}) \}_{k \in \nat}$, $i=1,\ldots,I$, of $\pos^2$-valued r.v.'s. For $l = 1,\ldots, Z_i^0$, $B_{i l}^0$ and $D_{i l}^0$ represent the residual size and residual patience time at time zero of initial flow~$l$ on route~$i$. For $k \in \nat$, $B_{i k}$ and $D_{i k}$ represent the initial size and initial patience time of flow~$k$ on route~$i$, where ``initial'' means as upon arrival at time~$U_{i k}$.  Let $(B_{i k}, D_{i k})$, $k \in \nat$, be i.i.d. copies of a r.v. $(B_i, D_i)$ with distribution law $\theta_i$; and let the mean values $\ex B_i =: 1/\mu_i$ and $\ex D_i = 1/ \nu_i$ be finite. Assume that the sequences $\{ (B_{i k}, D_{i k}) \}_{k \in \nat}$ are independent and do not depend on the arrival process $E(\cdot)$. For the moment, we do not make any specific assumptions about the sequences $\{ (B_{i l}^0, D_{i l}^0) \}_{l \in \nat}$.

\paragraph{State descriptor}
We denote the {\it population process} by $Z(\cdot)=(Z_1(\cdot), \ldots, Z_I(\cdot))$, where $Z_i(t)$ is the number of flows on route~$i$ at time~$t$. As can be seen from what follows, $Z(\cdot)$ is a~random element of the Skorokhod space $\skr_{\nneg \to \nneg^I}$.

For $i = 1,\ldots, I$, introduce operators $S_i \colon \skr_{\nneg \to \nneg^I} \to \mathbf{C}_{\nneg^2 \to \nneg}$ defined by
\[
S_i(z,s,t) := \int_s^t  \lmb_i(z(u)) du,
\]
For $t \geq s \geq 0$, $S_i(Z,s,t)$ is the {\it cumulative bandwidth allocated per flow on route~$i$ during time interval $[s,t]$}. The {\it residual size} and {\it residual lead time at time~$t$} of initial flow $l = 1, \ldots,  Z_i^0$ on route~$i$ are given by 
\[ 
B_{i l}^0(t) :=(B_{i l}^0 - S_i(Z,0,t))^+  \quad \text{and} \quad D_{i l}^0(t) :=(D_{i l}^0 - t)^+ ,
\]
and those of flow $k = 1, \ldots, E_i(t)$ on route~$i$ by
\[
B_{i k}(t) :=(B_{i k} - S_i(Z, U_{i k},t))^+ \quad  \text{and} \quad D_{i k}(t) :=(D_{i k} - (t - U_{i k} ) )^+.
\]

The state of the network at any time $t$ is defined by the residual sizes and residual patience times of the flows present in the network. With each flow, we associate a~dot in $\nneg^2$, whose coordinates are the residual size and residual patience time of the flow (see Fig.~\ref{fig:state_descriptor}). As a flow is getting transferred, the corresponding dot moves toward the axis: to the left at the transfer rate (which is $\lmb_i(Z(t))$ for a~flow on route~$i$) and downward at the constant rate of~$1$. As a dot hits the vertical axis, the corresponding flow leaves due to completion of its transfer. As a dot hits the horizontal axis, the corresponding flow leaves due to impatience. We combine these moving dots into the stochastic process $\clz(\cdot) \in \skr_{\nneg \to \msr^I}$ with
\begin{equation} \label{eq:state_descriptor}
\clz_i(t) :=\sum_{l=1}^{Z_i^0} \dlt^+_{(B_{i l}^0(t),D_{i l}^0(t))} + \sum_{k=1}^{E_i(t)} \dlt^+_{(B_{i k}(t),D_{i k}(t))},
\end{equation}
where, for $x \in \nneg^2$, $\dlt^+_x \in \msr$ is the Dirac measure at $x$ if $x_1 \wedge x_2 > 0$ and zero measure otherwise (i.e. assigns a zero mass to any Borel subset of $\nneg^2$). That is, $\clz_i(t)$ is a counting measure on~$\nneg^2$ that puts a unit mass to each of the dots representing class~$i$ flows except those on the axes. The process $\clz(\cdot)$ given by~\eqref{eq:state_descriptor} is called the {\it state descriptor}. Note that the total mass of the state descriptor coincides with the network population, $\la 1, \clz(\cdot) \ra = Z(\cdot)$.

When proving the results of the paper, we decompose the state descriptors into two parts keeping track of initial and newly arriving flows, respectively. That is,
\[
\clz(\cdot) = \clz^{\text{init}}(\cdot) + \clz^{\text{new}}(\cdot),
\]
where
\begin{align*}
\clz^{\text{init}}_i(t):= \sum_{l=1}^{Z_i(0)} \dlt^+_{(B_{il}^0(t), D_{il}^0(t))} \quad \text{and} \quad
\clz^{\text{new}}_i(t):= \sum_{k=1}^{E_i(t)} \dlt^+_{(B_{ik}(t), D_{ik}(t))}.
\end{align*}

We also define the corresponding total mass processes
\[
Z^{\text{init}}(\cdot) := \la 1, \clz^{\text{init}}(\cdot) \ra \quad \text{and} \quad Z^{\text{new}}(\cdot) := \la 1, \clz^{\text{new}}(\cdot) \ra.
\]

\begin{figure}[!htb] 
\centering
\includegraphics[scale=0.65]{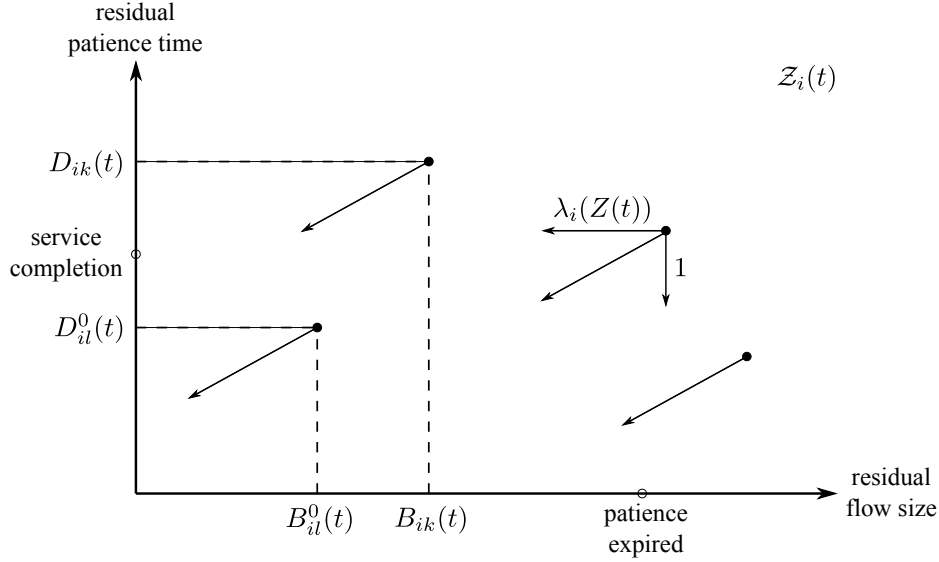}
\caption{The $i$-th coordinate $\clz_i(\cdot)$ of the state descriptor puts a unit mass to \\ the dots representing class~$i$ flows except those on the axes. }
\label{fig:state_descriptor}
\end{figure}

\section{Fluid model} \label{sec:fluid_model}
In this section we define and investigate a fluid model that is a deterministic analogue of the stochastic model described in the previous section. Later on the fluid model will be shown to arise as the limit of the stochastic model under a proper scaling. This convergence implies, in particular, existence of the fluid model.

To define the fluid model we need data $(\eta, \theta, \zeta^0) \in \pos^I \times \msr^I \times \msr^I$. The coordinates of $\eta$ play the role of arrival rates. As in the previous section, $\theta_i$ is the joint distribution of the generic size $B_i$ and patience time $D_i$ of a newly arrived flow on route~$i$ with finite expectations $\ex B_i = 1/\mu_i$ and $\ex D_i = 1/ \nu_i$. We also introduce the constants
\[
\rho_i := \eta_i / \mu_i, \quad \sgm_i := \eta_i / \nu_i,
\]
and the vectors $\rho, \sgm \in \pos^I$,
\[
\rho: = (\rho_1, \ldots, \rho_I), \quad \rho: = (\sgm_1, \ldots, \sgm_I).
\]
Finally, the measure-valued vector $\zeta^0$ characterizes the initial state of the network. Put $z^0 := \la 1, \zeta^0 \ra$ and, for all~$i$, take a r.v. $(B_i^0, D_i^0)$ that is degenerate at $(0,0)$ if $z_i^0 = 0$ and has distribution $\zeta_i^0 / z_i^0$ otherwise. Then $z^0$ represents the initial population, and $(B_i^0, D_i^0)$ the generic size and patience time of an initial flow on route~$i$.  We only consider initial conditions $\zeta^0$ such that the (marginal) distributions of $B_i^0$ and $D_i^0$ have no atoms. This restriction is necessary because we require the fluid model to be continuous, see Definition~\ref{def:FMS} below.

Denote by $\clc$ the collection of corner sets,
\[
\clc := \{ [x,\infty) \times [y,\infty) \colon (x,y) \in \nneg^2 \}.
\]

\begin{definition} \label{def:FMS} A pair $(\zeta, z) \in \mathbf{C}_{\nneg \to \msr^I} \times \mathbf{C}_{\nneg \to \nneg^I}$ is called a {\it fluid model solution (FMS) for the data $(\eta, \theta,\zeta^0)$} if $z(\cdot) = \la 1, \zeta(\cdot) \ra$ and, for all $i$, $t \geq 0$ and $A \in \clc$,
\begin{align} \label{eq:mvfms}
\zeta_i(t) (A) =& z_i^0 \pr \{ (B_i^0, D_i^0) \in A + (S_i(z,0,t), t) \}  \nonumber \\
&+ \eta_i \int_0^t \pr \{ (B_i,D_i) \in A + (S_i(z,s,t), t - s) \} ds.
\end{align}
In particular, for all $i$ and $t \geq 0$,
\begin{align} \label{eq:nfms}
z_i(t) = \zeta_i(t) (\nneg^2) =& z_i^0 \pr \{ B_i^0 \geq S_i(z,0,t), D_i^0 \geq t \}  \nonumber \\
&+ \eta_i \int_0^t \pr \{ B_i \geq S_i(z,s,t), D_i \geq t - s \} ds.
\end{align}
The function $\zeta(\cdot)$ is called a {\it measure-valued fluid model solution (MVFMS)} and the function $z(\cdot)$ a {\it numeric fluid model solution (NFMS)}
\end{definition}

Equations~\eqref{eq:mvfms} and~\eqref{eq:nfms} have appealing physical interpretations. For example,~\eqref{eq:nfms} simply means that a flow is still in the network at time~$t$ if its size and patience exceed, respectively, the amount of service it has received and the time that has passed since its arrival up to time~$t$.

\begin{remark}
By Dynkin's $\pi$-$\lmb$ theorem (see~\cite[Section 2.3]{GRZ08}), FMS's satisfy~\eqref{eq:mvfms} with any Borel set $A \subseteq \nneg^2$.
\end{remark}

\begin{remark} \label{rem:shifted_fms}
FMS's are invariant with respect to time shifts in the sense that, if $(\zeta, z)(\cdot)$ is an~FMS, then, for any $\dlt > 0$, $(\zeta^\dlt, z^\dlt)(\cdot) := (\zeta, z)(\cdot + \dlt)$ is an~FMS for the data $(\eta, \theta, \zeta(\dlt))$. That is, for all~$i$, $t \geq \dlt$ and Borel sets $A  \subseteq \nneg^2$,
\begin{subequations}
\begin{align}
\zeta_i(t)(A) &= \zeta_i(\dlt)(A+(S_i(z,\dlt,t), t - \dlt)) + \eta_i \int_\dlt^t \pr \{ (B_i, D_i) \in A+(S_i(z,s,t), t - s) \} ds, \label{eq:mvfms_shifted} \\
z_i(t) &= \zeta_i(\dlt) ([S_i(z,\dlt,t), \infty) \times [t - \dlt, \infty)) + \eta_i \int_\dlt^t \pr \{ B_i \geq S_i(z,s,t), D_i \geq t - s \} ds. \label{eq:fms_shifted}
\end{align}
\end{subequations}
\end{remark}

\begin{remark} \label{rem1}
The measure-valued and numeric components of an FMS uniquely define each other. In particular, uniqueness of an NFMS implies uniqueness of an MVFMS, and the other way around.
\end{remark}

As was mentioned earlier, the existence of FMS's is guaranteed by Theorem~\ref{th:fluid_limits} that follows in the next section. In the rest of this section, we discuss sufficient conditions for an~FMS to be unique and for an invariant (i.e. constant) FMS to be unique and asymptotically stable. To prove the stability result, we derive relations for asymptotic bounds for FMS's, which seems to be a novel approach since we have not seen analogous results in the related literature. We also give an example of multiple invariant FMS's.

\paragraph{Uniqueness of an FMS} The proof of the following theorem follows along the lines of the proofs of similar results~\cite[Proposition 4.2]{BEZ09} and~\cite[Theorem 3.5]{GRZ08}, see Section~\ref{sec:proofs_fluid_model}.
\begin{theorem} \label{th:unique_fms} Suppose that either \textup{(i)} $z_i^0 = 0 $ for all $i$, or \textup{(ii)} $z^0 \in \pos^I$ and the first projection of $\zeta^0$ is Lipschitz continuous, i.e. there exists a~constant $L \in \pos$ such that  for all $i$, $x < x'$ and $y$,
\[
\zeta_i^0 ([x,x'] \times [y,\infty)) \leq L (x' - x).
\]
Then an FMS for the data $(\eta, \theta, \zeta^0)$ is unique.
\end{theorem}

\paragraph{Uniqueness of an invariant FMS} Let $(\zeta,z)$ be an invariant FMS. By Lemma~\ref{lem:fms_bounded} in Section~\ref{sec:proofs_fluid_model}, all of the coordinates of $z$ are positive, and the fluid model equations~\eqref{eq:mvfms} and~\eqref{eq:nfms} for $(\zeta,z)$ look as follows: for all~$i$, Borel subsets $A \subseteq \nneg^2$ and $t \geq 0$,
\begin{align}
\zeta_i(A) &= \zeta_i (A + (\lmb_i(z)t,t)) + \eta_i \int_0^t \theta_i(A + (\lmb_i(z)s,s)) ds, \label{eq:fixed_msr_1} \\
z_i &= \zeta_i([\lmb_i(z)t, \infty) \times [t, \infty)) + \eta_i \int_0^t \pr \{ B_i \geq \lmb_i(z)s, D_i \geq s \} ds \label{eq:fixed_point_1}.
\end{align}
Letting $t \to \infty$ in~\eqref{eq:fixed_msr_1} and~\eqref{eq:fixed_point_1}, we obtain the equations
\begin{align}
\zeta_i(A) &= \eta_i \int_0^\infty \theta_i(A + (\lmb_i(z)s,s)) ds, \label{eq:fixed_msr}\\
z_i &= \eta_i \ex (B_i/\lmb_i(z) \wedge D_i ), \label{eq:fixed_point}
\end{align}
which are actually equivalent to~\eqref{eq:fixed_msr_1} and~\eqref{eq:fixed_point_1}.

Thus, we have the closed-form equation~\eqref{eq:fixed_point} for the numeric components of invariant FMS's, and the corresponding measure-valued components are defined by~\eqref{eq:fixed_msr}. In particular, uniqueness of an invariant FMS is equivalent to uniqueness of a solution to~\eqref{eq:fixed_point}.

Multiplying the coordinates of~\eqref{eq:fixed_point} by the corresponding rates $\lmb_i(z)$, we obtain the equivalent equation
\begin{equation} \label{eq:fixed_point_equiv}
\Lmb_i(z) = g_i(\lmb_i(z)) \quad \text{for all $i$,}
\end{equation}
where 
\[
g_i(x) := \eta_i \ex (B_i \wedge x D_i), \quad x \geq 0.
\]

We suggest a sufficient condition for uniqueness of a solution to~\eqref{eq:fixed_point_equiv} (i.e. of an invariant NFMS) that involves the left most points of supports of certain distributions.

\begin{definition}
For an $\real$-valued r.v. $X$, denote by $\inf X$ the left most point of its support. Recall that the {\it support} of $X$  is the minimal (in the sense of inclusion) closed interval $S$ such that $\pr \{ X \in S \} = 1$.
\end{definition}

As we show later (see Lemmas~\ref{lem:g_increasing} and~\ref{lem:alpha} in Section~\ref{sec:proofs_fluid_model}), if $m_i \leq 1/\inf (D_i/B_i)$, where $1/0 := \infty$ by definition, the function $g_i(\cdot)$ is continuous and strictly increasing in the interval $[0,m_i]$, implying that its inverse is well-defined in $[0, g_i(m_i)]$. Then we can prove the following result.

\begin{theorem} \label{th:unique_fixed_point}
Let 
\begin{equation} \label{eq:suf_cond}
\inf (D_i/B_i) \leq 1/m_i \quad \text{for all $i$}.
\end{equation}
Then there exists a unique invariant FMS $(\zeta^\ast,z^\ast)$, and the bandwidth allocation vector $\Lmb(z^\ast)$ is the unique solution to the optimization problem
\begin{equation} \label{eq:th_unique_fixed_point}
\text{maximize} \quad \sum_{i = 1}^I G_i( \Lambda_i) \quad \text{subject to} \quad A\Lambda \leq C, \quad \Lambda_i \leq g_i(m_i) \quad \text{for all $i$},
\end{equation}
with strictly concave functions $G_i(\cdot)$ such that $G_i'(\cdot) = \mathcal{U}'_i (g_i^{-1}(\cdot))$ in $[0, g_i(m_i)]$.
\end{theorem}

\begin{remark}
Note that it is realistic to assume that flows do not abandon if they are always served at the maximum rate, i.e. that $D_i \geq B_i / m_i$. For such routes, we have $g_i(m_i) = \rho_i$, and the sufficient uniqueness condition~\eqref{eq:suf_cond} reads as $\inf (D_i / B_i) = 1/m_i$.
\end{remark}

The complete proof of Theorem~\ref{th:unique_fixed_point} is postponed to Section~\ref{sec:proofs_fluid_model}, here we only discuss the main ideas. For now, let $z^\ast$ be any invariant NFMS. As we plug the fixed point equation~\eqref{eq:fixed_point_equiv} into the optimization problem~\eqref{eq:def_lmb} for the rate vector $\lmb(z^\ast)$, the problem~\eqref{eq:th_unique_fixed_point} follows, which is strictly concave and does not depend on~$z^\ast$. Hence, $\Lmb(z^\ast) =: \Lambda^\ast$ is the same for all invariant points $z^\ast$. This idea of transforming the optimization problem defining the rate vector combined the fixed point equation into an independent problem for the bandwidth allocation vector we adopted from Borst {\it et al.} \cite[Lemma 5.2]{BEZ09}. Now, since the functions $g_i(\cdot)$ are invertible in the feasible rate intervals $[0,m_i]$, it follows from~\eqref{eq:fixed_point_equiv} that the fixed point is unique and given by 
\begin{equation} \label{eq:fp_via_Lmb}
z^\ast_i = \Lmb^\ast_i / g_i^{-1}(\Lmb^\ast_i).
\end{equation}

Note that this method not only proves uniqueness of an invariant FMS, but also suggests a two-step algorithm to compute it: first we need to solve the strictly concave optimization problem~\eqref{eq:th_unique_fixed_point} for $\Lmb^\ast$, which can be done with any desired accuracy in a polynomial time, and then we can compute the fixed point $z^\ast$ itself by formula \eqref{eq:fp_via_Lmb}.

\paragraph{Asymptotic bounds for FMS's} Here we derive asymptotic bounds for NFMS's that, for a wide class of bandwidth-sharing networks, imply convergence to the invariant NFMS provided it is unique.
\begin{theorem} \label{th:stable_fixed_point}
There exist constants $l, u \in \pos^I$ such that, for any NFMS $z(\cdot)$,
\[
0 < l_i \leq \liminf_{t \to \infty} z_i(t) \leq \limsup_{t \to \infty} z_i(t) \leq u_i \quad \text{for all $i$}.
\]
These constants satisfy the relations
\begin{equation} \label{eq:bounds_fms}
l_i = \eta_i \ex ( B_i / R_i(l,u) \wedge D_i ), \quad u_i = \eta_i \ex ( B_i / r_i(l,u) \wedge D_i ) \quad \text{for all $i$},
\end{equation}
where the functions $r(\cdot,\cdot)$ and $R(\cdot,\cdot)$ are defined by
\[
r_i(x,x') := \inf_{x \leq z \leq x'} \lmb_i(z), \quad R_i(x,x') := \sup_{x \leq z \leq x'} \lmb_i(z) \quad  \text{for all $i$ and $x \leq x'$}.
\]
\end{theorem}

\begin{remark}
There could be more than one pair $(l,u)$ solving~\eqref{eq:bounds_fms}. The asymptotic bounds $l$~and~$u$ for NFMS's given by Theorem~\ref{th:stable_fixed_point} form one of such pairs.
\end{remark}

We now proceed with the proof of Theorem~\ref{th:stable_fixed_point}.

\begin{proof}
Note that if
\begin{equation} \label{eq37}
0 < \tilde{l}_i \leq \liminf_{t \to \infty} z_i(t) \leq \limsup_{t \to \infty} z_i(t) \leq \tilde{u}_i \quad \text{for all $i$},
\end{equation}
then also
\begin{equation} \label{eq38}
\eta_i \ex (B_i / R_i(\tilde{l},\tilde{u}) \wedge D_i ) \leq \liminf_{t \to \infty} z_i(t) 
\leq \limsup_{t \to \infty} z_i(t) \leq \eta_i \ex ( B_i / r_i(\tilde{l},\tilde{u}) \wedge D_i ) \quad \text{for all $i$}.
\end{equation}
Indeed, by~\eqref{eq37}, for any $\eps \in (0,\min_{1 \leq i \leq I} \tilde{l}_i)$, there exists a~$t_\eps$ such that
\[
\tilde{l}_i - \eps \leq z_i(t) \leq \tilde{u}_i + \eps \quad \text{for all $i$ and $t \geq t_\eps$}.
\]
Introduce the vectors
\[
\tilde{l} - \eps := (\tilde{l}_1 - \eps, \ldots, \tilde{l}_I - \eps), \quad \tilde{u} + \eps := (\tilde{u}_1 + \eps, \ldots, \tilde{u}_I + \eps).
\]
Then
\[
r_i(\tilde{l} - \eps,\tilde{u} + \eps) (t-s) \leq S_i(z,s,t) \leq R_i(\tilde{l} - \eps,\tilde{u} + \eps) (t-s) \quad \text{for $t \geq s \geq t_\eps$},
\]
which, when plugged into the shifted fluid model equation~\eqref{eq:fms_shifted}, implies that
\begin{align*}
z_i(t) \geq& \ \eta_i \int_{t_\eps}^t \mathbb{P} \{ B_i \geq R_i(\tilde{l} - \eps,\tilde{u} + \eps) (t-s), D_i \geq (t-s) \} ds, \\ 
z_i(t) \leq& \ \zeta_i(t_\eps) \bigl( [S_i(z,t_\eps,t),\infty) \times [t - t_\eps, \infty) \bigr) \\
&+ \eta_i \int_{t_\eps}^t \mathbb{P} \{ B_i \geq r_i(\tilde{l} - \eps,\tilde{u} + \eps) (t-s), D_i \geq (t-s) \} ds \quad \text{for $t \geq t_\eps$},
\end{align*}
where $\zeta(\cdot)$ is the corresponding MVFMS. Taking $t \to \infty$ in the last two inequalities, we obtain
\[
\eta_i \ex  ( B_i / R_i(\tilde{l} - \eps,\tilde{u} + \eps) \wedge D_i ) \leq \liminf_{t \to \infty} z_i(t) \leq \limsup_{t \to \infty} z_i(t) \leq \eta_i \ex  ( B_i / r_i(\tilde{l} - \eps,\tilde{u} + \eps) \wedge D_i ),
\]
and then~\eqref{eq38} follows as $\eps \to 0$.

Now we will iterate \eqref{eq37}--\eqref{eq38}. The rate constraints plugged into~\eqref{eq:nfms} imply the initial bounds
\[
0 < l_i^0 :=  \eta_i \ex ( B_i / m_i \wedge D_i ) \leq \liminf_{t \to \infty} z_i(t) \leq \limsup_{t \to \infty} z_i(t) \leq \eta_i \ex D_i =: u_i^0 \quad \text{for all $i$},
\]
and then \eqref{eq37}--\eqref{eq38} yield the recursive bounds 
\begin{equation} \label{eq39}
\begin{split}
l_i^k &:=  \eta_i \ex ( B_i / R_i(l^{k-1},u^{k-1}) \wedge D_i ) \leq \liminf_{t \to \infty} z_i(t), \\
u_i^k &:=  \eta_i \ex ( B_i / r_i(l^{k-1},u^{k-1}) \wedge D_i ) \ \geq \limsup_{t \to \infty} z_i(t) \quad \text{for all $k \in \nat$ and $i$}.
\end{split}
\end{equation}

The sequence $\{ l^k \}_{k \in \nat}$ is non-decreasing and bounded from above by~$u^0$. The sequence $\{ u^k \}_{k \in \nat}$ is non-increasing and bounded from below by~$l^0$. Hence, there exist the limits $\lim l^k =: l$ and $\lim u^k =: u$. In~\eqref{eq39}, let $k \to \infty$, then~\eqref{eq:bounds_fms} follows.

Note finally that the recursive bounds $\{ l^k \}_{k \in \nat}$ and $\{ u^k \}_{k \in \nat}$ as well as their limits $l$ and $u$ do not depend on a particular NFMS. 
\end{proof}

\paragraph{Asymptotic stability of an invariant fluid model solution} It is tractable to assume that transfer rates in a bandwidth-sharing network decrease as its population grows. In particular, tree networks satisfy this property, see~\cite{BEZ09}.
\begin{definition}
If $z' \geq z \in \pos^I$ implies $\lmb(z') \leq \lmb(z)$, the network is called {\it monotone}.
\end{definition}

For monotone networks, the system of equations~\eqref{eq:bounds_fms} decomposes into two independent systems of equations for the lower bound~$l$ and for the upper bound~$u$:
\begin{subequations}
\begin{align}
l_i &= \eta_i \ex (B_i/\lmb_i(l) \wedge D_i ) \quad \text{for all~$i$},   \label{eq:mntn_lb} \\
u_i &= \eta_i \ex (B_i/\lmb_i(u) \wedge D_i ) \quad \text{for all~$i$} \label{eq:mntn_ub},
\end{align}
\end{subequations}
which implies the following result.
\begin{corollary} \label{cor:stable_fixed_point}
Suppose that the network is monotone and has a unique invariant FMS $(\zeta^\ast, z^\ast)$. Then any FMS $(\zeta,z)(t) \to (\zeta^\ast, z^\ast)$ as $t \to \infty$.
\end{corollary}
Indeed, both~\eqref{eq:mntn_lb} and~\eqref{eq:mntn_ub} coincide with the fixed point equation~\eqref{eq:fixed_point}, and since Corollary~\ref{cor:stable_fixed_point} assumes that the latter equation has a unique solution $z^\ast$,  it immediately follows by Theorem~\ref{th:stable_fixed_point} that, for any NFMS $z(\cdot)$,
\[
l_i = \liminf_{t \to \infty} z_i(t) = \limsup_{t \to \infty} z_i(t) = u_i = z_i^\ast \quad \text{for all $i$}.
\]
In the Appendix we also show that $z(t) \to z^\ast$ implies $\zeta(t) \to \zeta^\ast$.

\paragraph{Example: Single Link} The sufficient condition for uniqueness of an invariant FMS given by Theorem~\ref{th:unique_fixed_point} is sometimes also necessary. Consider, for example, processor sharing in critical load, that is $J=I=1$ and (omitting the link and class indices) $\rho = C$. In this case, the fixed point equation~\eqref{eq:fixed_point} looks like
\[
z = \eta \ex \left( \frac{B}{C/z \wedge m} \wedge D \right),
\]
which, for $z$ such that $C/z \leq m$ and $B z/C \leq D$ a.s., reduces to
\[
z = \eta \ex (B z/C \wedge D) = \eta \ex B z/C = \rho z /C = z.
\]
I.e. any $z \in [C/m, C \inf D/B]$ is an invariant NFMS. In particular, if $\inf D/B > 1/m$, which violates the assumption of Theorem~\ref{th:unique_fixed_point}, then there is a continuum of invariant FMS's.

For a single link critically loaded by multiple classes of flows, we have an analogous result, which is more complicated to derive and therefore the proof is postponed to Section~\ref{sec:proofs_fluid_model}.

\begin{theorem} \label{th:single_link}
Assume that $J=1$ (in what follows we omit the link index), and that the utility functions are $\clu_i(x) = \kappa_i \log x$. If $\sum_{i=1}^I \eta_i \ex (B_i/m_i \wedge D_i) \neq C$, then there is a unique invariant FMS. Otherwise there might be a continuum of invariant FMS's.
\end{theorem}

\section{Sequence of stochastic models and fluid limit theorem} \label{sec:fluid_limits}
In this section we study the asymptotic behavior of the stochastic network described in Section~\ref{sec:stochastic_model} as its global parameters --- capacities and arrival rates --- grow large, while the characteristics of an individual flow remain of a fixed order. We refer to this scaling as the large capacity regime.

\paragraph{Large capacity scaling} To a~sequence $\clr$ of positive numbers increasing to $\infty$, we associate a sequence of stochastic models as defined in Section~\ref{sec:stochastic_model}. We mark all parameters associated with the $r$-th model with a~superscript~$r$ and assume the following:
\begin{itemize}
\item[(A.1)] $\phantom{t}$ network structure, rate constraints and utility function are the same in all models: $\phantom{tt}$~$A^r = A$,~$m^r = m$ and $\clu_i^r(\cdot) = \clu_i(\cdot)$ for all~$i$;
\item[(A.2)] $\phantom{t}$ link capacities grow linearly in $r$: $C^r = r C$;
\item[(A.3)] $\phantom{t}$ arrival rates grow linearly in $r$: $\ove^r(\cdot) := E^r(\cdot)/r \Rightarrow \eta(\cdot)$ as $r \to \infty$, 
where $\eta(t) := t \eta$ and $\phantom{tt}$ $\eta \in (0,\infty)^I$;
\item[(A.4)] $\phantom{t}$ flow sizes and patience times remain of a fixed order: for all~$i$, $(B_i^r, D_i^r) \Rightarrow (B_i, D_i)$, $\phantom{tt}$ where $(B_i,D_i)$ are $\pos^I$-valued r.v.'s with distributions $\theta_i$ and finite mean values $\phantom{tt}$~$(1/\mu_i, 1/\nu_i)$, and also $(1/\mu_i^r, 1/\nu_i^r) \to (1/\mu_i, 1/\nu_i)$;
\item[(A.5)] $\phantom{t}$ the scaled initial configuration converges in distribution to a random vector of finite $\phantom{tt} \ $measures: $\cloz^r(0) :=  \clz^r(0)/r \Rightarrow \zeta^0$;
\item[(A.6)] $\phantom{t}$ the projections $\zeta^0_i(\cdot \times \nneg)$ and $\zeta^0_i(\nneg \times \cdot)$ are a.s. free of atoms for all~$i$.
\end{itemize}
\vspace{2pt}

\paragraph{Fluid limit theorem} In the large capacity regime, the stochastic model defined is Section~\ref{sec:stochastic_model} converges to the fluid model defined in Section~\ref{sec:fluid_model}. More precisely, introduce the scaled state descriptors and population processes
\[
\cloz^r(\cdot) := \clz^r(\cdot) / r, \quad \oz^r(\cdot) := \la 1, \cloz^r(\cdot) \ra = Z^r(\cdot) / r.
\]
Also introduce the scaled versions of the two components of the state descriptor:
\begin{equation*}
\begin{alignedat}{2}
\clozinit(\cdot) &:= \clzinit(\cdot) / r,& \quad \quad \quad \ozinit(\cdot) &:= \la 1, \clozinit(\cdot) \ra = \zinit(\cdot) / r,\\
\cloznew(\cdot) &:= \clznew(\cdot) / r,& \quad \quad \quad \oznew(\cdot) &:= \la 1, \cloznew(\cdot) \ra = \znew(\cdot) / r.
\end{alignedat}
\end{equation*}

\begin{remark} Let $\lmb(\cdot)$ be the rate allocation function in the unscaled network, then
\begin{align*}
\lmb^r(z) =: & \, \argmax_{\begin{subarray}{l} A (\lmb \ast z) \leq r C \\ \lmb \leq m \end{subarray}} \sum\nolimits_{i=1}^I z_i \, \mathcal{U}_i( \lmb_i)
\\
= & \, \argmax_{\begin{subarray}{l} A (\lmb \ast z/r) \leq C \\ \lmb \leq m \end{subarray}} \sum\nolimits_{i=1}^I (z_i/r) \, \mathcal{U}_i( \lmb_i) =: \lmb(z / r),
\end{align*}
and
\[
S_i^r(Z^r,s,t) := \int_s^t \lmb_i^r( Z^r(u) )du = \int_s^t \lmb_i( \oz^r(u) )du =: S_i(\oz^r,s,t).
\]
\end{remark}

We now provide the definition of a fluid limit followed by the main result of this section.
 
\begin{definition}
We refer to weak limits along convergent subsequences $\{ (\cloz^q, \oz^q)(\cdot) \}_{q \in \clq}$, $\clq \subseteq \clr$, as {\it fluid limits}.
\end{definition}

\begin{theorem} \label{th:fluid_limits}
Under the assumptions \textup{(A.1)--(A.6)}, the sequence $\{ (\cloz^r, \oz^r)(\cdot) \}_{r \in \clr}$ is tight in $\skr_{\nneg \to \msr^I} \times \skr_{\nneg \to \nneg^I}$, and all fluid limits are a.s. FMS's for the data $(\eta, \theta, \zeta^0)$. In particular, if~there is a unique FMS $(\clz,Z)(\cdot)$ for the data $(\eta, \theta, \zeta^0)$, then $(\cloz^r, \oz^r)(\cdot) \Rightarrow (\clz,Z)(\cdot)$ as $r \to \infty$.
\end{theorem}

The proof follows in Section~\ref{sec:proof_fluid_limits}. To show tightness we adjust the techniques of~\cite{GW09} to the two-dimensional case, since in~\cite{GW09} flows are patient and state descriptors are vectors of measures on $\nneg$. The proof of convergence to FMS's follows the lines of that in~\cite{GRZ08}. It uses the boundedness of fluid limits away from zero, and the key difference is that in~\cite{GRZ08} this property is guaranteed by the overload regime, while in our model it holds in any load regime due to the rate constraints.

\section{Convergence of stationary distributions} \label{sec:stat_distributions}
Assume that, in the stochastic model defined in Section~\ref{sec:stochastic_model}, the arrival processes are Poisson of rates $\eta_1, \ldots, \eta_I$. Then there exists a unique stationary (and also limiting as $t \to \infty$) distribution of the state descriptor $\clz(\cdot)$. Indeed, without loss of generality, there are i.i.d. r.v.'s $\{  \widetilde{D}_{ik} \}_{k \in \nat, 1\leq i \leq I}$ distributed as $\max_{1 \leq i \leq I} D_i$ and such that a.s. $D_{ik} \leq \widetilde{D}_{ik}$ for all $k$ and $i$. Then the total population $\sum_{i=1}^I Z_i(\cdot)$ of the network is a.s. and within the whole time horizon $\nneg$ bounded from above by the length of the $M/G/\infty$ queue with the following parameters. At time $t = 0$, there are $\sum_{i=1}^I Z_i(0)$ customers in the queue whose service times are patience times of initial flows in the network. The input process for the queue is the composition of those for the network, and hence is Poisson of rate $\sum_{i=1}^I \eta_i$. Service times of new customers in the queue are drawn from the sequence $\{  \widetilde{D}_{ik} \}_{k \in \nat, 1\leq i \leq I}$ of upper bounds for patience times of new flows in the network. As any other $M/G/\infty$ queue, the defined queue is regenerative. The instants when a customer enters the empty queue form an embedded renewal process whose cycle length is non-lattice and has a finite mean value $\exp(\sum_{i=1}^I \eta_i \ \ex \widetilde{D}_{11}) / \sum_{i=1}^I \eta_i$. With respect to this renewal process , the state descriptor $\clz(\cdot)$ is also regenerative. Then, by~\cite[Chapter~V.I, Theorem~1.2]{Asmussen}, there exists a limiting distribution for $\clz(\cdot)$.

Now consider a sequence of stochastic models as defined in Section~\ref{sec:stochastic_model} that satisfies the assumptions (A.1), (A.2), (A.4) (see Section~\ref{sec:fluid_limits}) and
\begin{itemize}
\item[(A$'$.3)] $\phantom{t}$ the input processes $E_1^r(\cdot), \ldots, E_I^r(\cdot)$ are independent Poisson processes of rates $\eta_1^r, \ldots, \eta_I^r$, $\phantom{tt}$ and  
$\eta^r/r \to \eta \in\pos^I$ as $r \to \infty$.
\end{itemize}

Let $\cly^r$ have the stationary distribution of $\clz^r(\cdot)$ and put $Y^r := \la 1, \cly^r \ra$. Introduce also the scaled versions
\[
\cloy^r := \cly^r / r, \quad \oy^r := \la 1, \cloy^r \ra = Y^r /r.
\]

We now have the following result.

\begin{theorem} \label{th:stationary_distributions}
Under the assumptions \textup{(A.1), (A.2), (A$'$.3)} and \textup{(A.4)}, the sequence $\{(\cloy^r, \oy^r)\}_{r \in \clr}$ is tight, and any weak limit point $(\cly,Y)$ is a weak invariant FMS, i.e. there exists a stationary FMS $(\clz,Z)(t) \stackrel{\text{d}}{=} (\cly,Y)$, $t \geq 0$. In particular, by Corollary~\ref{cor:stable_fixed_point}, if the network is monotone and has a unique invariant FMS $(\zeta_\ast, z_\ast)$, then $(\cloy^r, \oy^r) \Rightarrow (\zeta_\ast, z_\ast)$ as $r \to \infty$.
\end{theorem}

The general strategy of the proof is adopted from~\cite[Theorem~3.3]{KR10}: we check that any convergent~subsequence of initial conditions $\{\cloz^q(0) \stackrel{\text{d}}{=} \cloy^q\}_{q \in \clq}$, $\clq \subseteq \clr$, $\cloy^q \Rightarrow \cly$, satisfies the assumptions of the fluid limit theorem (we only need to check \textup{(A.6)}). Then the corresponding subsequence $\{ \cloz^q(\cdot) \}_{q \in \clq}$ of the scaled state descriptors converges to an MVFMS that is stationary (i.e. $\cly$ is a~weak invariant MVFMS) since all $\cloz^q(\cdot)$ are stationary. 

The techniques we use to implement this strategy are different from the techniques of~\cite{KR10}, though. Our key instruments for establishing tightness are $M/G/\infty$ bounds, see Section~\ref{sec:proof_stat_distributions}. Below we present an elegant proof of \textup{(A.6)} for weak limit points of $\{\cloy^r\}_{r \in \clr}$.

\begin{lemma} \label{lem:stat_no_atoms}
Any weak limit point $\cly$ of  $\{\cloy^r\}_{r \in \clr}$ has both projections $\cly(\cdot \times \nneg)$ and $\cly(\nneg \times \cdot)$ a.s. free of atoms.
\end{lemma}

\begin{proof} The key idea is the following. Consider the network in its stationary regime. Then, on one hand, it always has the same distribution, and on the other hand, all initial flows are gone at some point, and newly arriving flows do not accumulate along horizontal and vertical lines.

Let $\cly$ be the weak limit along a subsequence $\{\cloy^q\}_{q \in \clq}$, and run the $q$-th network starting from $\cloz^q(0) \stackrel{\text{d}}{=} \cloy^q$. By~\cite[Lemma~6.2]{GRZ08}, it suffices to show that, for any $\dlt > 0$ and $\eps > 0$, there exists an $a > 0$ such that
\begin{equation} \label{eq35}
\liminf_{q \to \infty} \pr^q \{ \sup_{x \in \nneg}\| \cloy^q(H_x^{x+a}) \| \vee \| \cloy^q(V_x^{x+a}) \| \leq \dlt \} \geq  1 - \eps,
\end{equation}
where $H_a^b := \nneg \times [a,b]$ and $V_a^b := [a,b] \times \nneg$ for all $b \geq a \geq 0$.

First we estimate the time when there is only a few (when scaled) initial flows left. The initial flows whose initial patience times are less than $t$ are already gone at time~$t$. Then Lemma~\ref{lem2} (see Section~\ref{sec:proof_stat_distributions}) implies that (recall that $\Pi(\lmb)$ stands for the Poisson distribution with parameter~$\lmb$)
\begin{align*}
\oz_i^{q,\, \text{init}}(t) \leq \cloz_i^q(0)(\nneg \times [t,\infty)) &\stackrel{\text{d}}{=}\cloy_i^q(\nneg \times [t,\infty)) \\
&\leq_{\text{st}} \frac{1}{q} \Pi(\eta_i^q \int_t^\infty \pr^q \{ D_i^q > y \} dy) \Rightarrow \eta_i \int_t^\infty \pr \{ D_i > y \} dy.
\end{align*}
Take $T$ such that $\max_{1 \leq i \leq I} \eta_i \int_t^\infty \pr \{ D_i > y \} dy < \dlt /2$, then
\begin{equation*}
\lim_{q \to \infty} \pr^r \{ \| \oz^{q,\, \text{init}}(T) \| \leq \dlt /2 \} = 1.
\end{equation*}
Now, in Lemma~\ref{lem:as_reg_znew} (see Section~\ref{sec:proof_fluid_limits}), we prove that newly arriving customers do not accumulate in thin horizontal~and vertical strips, i.e. there exists an $a>0$ such that
\begin{equation*}
\liminf_{q \to \infty} \pr^q \{ \sup_{t \in [0,T]} \sup_{x \in \nneg}\| \cloz^{q, \, \text{new}}(t)(H_x^{x+a}) \| \vee \| \cloz^{q, \, \text{new}}(t)(V_x^{x+a}) \| \leq \dlt/2\} \geq  1 - \eps.
\end{equation*}
Finally, because of stationarity of $\cly^q$,
\begin{align*}
&\, \pr^q \{ \sup_{x \in \nneg}\| \cloy^q(H_x^{x+a}) \| \vee \| \cloy^q(V_x^{x+a}) \| \leq \dlt \} \\
=& \, \pr^q \{ \sup_{x \in \nneg}\| \cloz^q(T)(H_x^{x+a}) \| \vee \| \cloz^q(T)(V_x^{x+a}) \| \leq \dlt \} \\
\geq& \, \pr^q\{ \|\oz^{q, \, \text{init}}(T) \| \leq \dlt/2, \sup_{x \in \nneg}\| \cloz^{q,\, \text{new}}(T)(H_x^{x+a}) \| \vee \| \cloz^{q,\, \text{new}}(T)(V_x^{x+a}) \| \leq \dlt/2 \},
\end{align*}
which implies \eqref{eq35} by the choice of $T$ and $a$.
\end{proof}

\section{Proof of fluid model properties} \label{sec:proofs_fluid_model}
Here we prove the results of Section~\ref{sec:fluid_model}.
\subsection{Proof of Theorem~\ref{th:unique_fms}}
In the proof of Theorem~\ref{th:unique_fms}, we exploit boundedness of NFMS's away from zero and in the norm (see Lemma~\ref{lem:fms_bounded}), and Lipschitz continuity of MVFMS's in the first coordinate (see Lemma~\ref{lem:fms_LC}). We also use the auxiliary Lemma \ref{lem1}. 

Recall the notations $\sgm_i = \eta_i \ex D_i$ and $\sgm = (\sgm_1, \ldots, \sgm_I)$.

\begin{lemma} \label{lem:fms_bounded}
Let $z(\cdot)$ be an NFMS Then $\sup_{t \geq 0} \| z(t) \| \leq \| z(0) \| + \| \sgm \| < \infty$ and, for any $\dlt > 0$, $\inf_{t \geq \dlt}\min_{1 \leq i \leq I} z_i(t) > 0$. In particular, if $z_i(0)> 0$, then $\inf_{t \geq 0} z_i(t) > 0$.
\end{lemma}

\begin{proof}
By the rate constraints, $S_i(z,s,t) \leq m_i(t - s)$ for all $s \leq t$, which, when plugged into the fluid model equation~\eqref{eq:nfms}, implies the following lower bound: $z_i(t) \geq \eta_i \int_0^t \pr \{ B_i / m_i \wedge D_i \geq s \} ds$. Since $f_i(s) := \pr \{ B_i / m_i \wedge D_i \geq s \} \uparrow \pr \{ B_i / m_i \wedge D_i > 0 \} = 1$ as $s \downarrow 0$, in a~small enough interval $[0,\eps]$, $f_i(\cdot) \geq 1/2$. Then, for $t \geq \dlt$, $z_i(t) \geq \eta_i \int_0^{\dlt \wedge \eps} f_i(s) ds \geq \eta_i (\dlt \wedge \eps) /2$. The upper bound follows from~\eqref{eq:nfms} directly: $z_i(t) \leq z_i(0) + \eta_i \int_0^t \pr \{ D_i \geq s \} ds \uparrow z_i(0) + \sgm_i$ as $t \uparrow \infty$.
\end{proof}

\begin{lemma} \label{lem1}
For an $\real$-valued r.v. $\xi$ and $x \leq x'$, $\int_\real \pr \{ u+x \leq \xi \leq u+x' \} du \leq x' - x$.
\end{lemma}

See the Appendix for the proof.

\begin{lemma} \label{lem:fms_LC}
Under assumption \textup{(ii)} of Theorem~\ref{th:unique_fms}, any MVFMS $\zeta(\cdot)$ at any time $t \geq 0$ has a~Lipschitz continuous first projection, i.e. there exists a~constant $L(\zeta,t) \in \pos$ such that  for all $i$, $x < x'$ and $y$,
\[
\zeta_i(t) ([x,x'] \times [y,\infty)) \leq L(\zeta,t) (x' - x).
\]
\end{lemma}

\begin{proof}
For an FMS $(\zeta, z)(\cdot)$, for all $i$, $t \geq 0$, $x < x'$ and $y$, 
\[
\zeta_i(t)([x,x'] \times [y,\infty)) \leq f_i(x,x',y) + \eta_i g_i(x,x',y),
\]
where
\begin{align} \label{eq1}
f_i(x,x',y) &:= \zeta_i^0 ( [ x + S_i(z,0,t), x' + S_i(z,0,t) ] \times [y + t, \infty) ), \nonumber \\
g_i(x,x',y) &:= \int_0^t \pr \{ x + S_i(z,s,t) \leq B_i \leq x' + S_i(z,s,t) \} ds.
\end{align}
By Lipschitz continuity of the initial condition, $f_i(x,x',y) \leq L (x' - x)$. In \eqref{eq1}, change the variable of integration for $v = V(s) := S_i(z,s,t)$. Then
\[
g_i(x,x',y) = \int_0^{S_i(z,0,t)} \pr \{ x + v \leq B_i \leq x' + v \} / \lmb_i(z(V^{-1}(v))) dv \leq M(\zeta,t) (x' - x),
\]
where $M(\zeta,t) := \sup_{s \in [0,t]} \max_{1 \leq i \leq I} 1 / \lmb_i(z(s))$. By Lemma~\ref{lem:fms_bounded}, the functions $1 / \lmb_i(z(\cdot))$ are continuous in $[0,t]$. Hence $M(\zeta,t)$ is finite and the first projection of $\zeta(t)$ is Lipschitz continuous with the constant $L(\zeta,t): = L + \| \eta \| M(\zeta,t)$.
\end{proof}

Now we are in a position to prove Theorem~\ref{th:unique_fms}.

{\it Proof of Theorem~\ref{th:unique_fms}.} Let $(\zeta^1, z^1)(\cdot)$ and $(\zeta^2, z^2)(\cdot)$ be two FMS's for the data $(\eta, \theta, \zeta^0)$. 

(i) We show that the two FMS's coincide in an interval $[0,\dlt]$. We check that $z^1(\dlt) = z^2(\dlt) \in \pos^I$ and that the first projection of $\zeta^1(\dlt) = \zeta^2(\dlt)$ is Lipschitz continuous. Then, by Remark~\ref{rem:shifted_fms} and the second part of the theorem, the two FMS's coincide everywhere.

Note that, for a vector $z \in \pos^I$ of a small enough norm, $\lmb_i(z) = m_i$ for all~$i$. Lemma~\ref{lem:fms_bounded} and the fluid model equation~\eqref{eq:nfms} imply that $0 < z_i^1(t), z_i^2(t) \leq \eta_i t$ for all $i$ and $t > 0$. Then, for all~$i$ and $s, t \in [0,\dlt]$, where $\dlt$ is small enough,
\begin{equation} \label{eq2}
S_i(z^1,s,t) = S_i(z^2,s,t) = m_i(t - s).
\end{equation}
Plugging~\eqref{eq2} into~\eqref{eq:nfms}, we obtain, for $t \in [0,\dlt]$ and all $i$, 
\[
z_i^1(t) = z_i^2(t) = \eta_i \int_0^t \pr \{ B_i / m_i \wedge D_i \geq s\} ds.
\]
By Remark \ref{rem1}, $\zeta^1(\cdot)$ and $\zeta^2(\cdot)$ coincide in $[0,\dlt]$, too. Lipschitz continuity of the first projection of $\zeta^1(\dlt) = \zeta^2(\dlt)$ follows as we plug~\eqref{eq2} into the fluid model equation~\eqref{eq:mvfms} (recall that it is valid for all Borel sets): for all $i$, $x < x'$ and~$y$,
\begin{align*}
\zeta_i^j(\dlt)([x,x'] \times [y,\infty)) &= \eta_i \int_0^\dlt \pr \{ x+m_i s \leq B_i \leq x'+ m_i s, D_i \geq y + s \} ds  \\
&\leq \eta_i \int_0^\dlt \pr \{ x/m_i + s \leq B_i / m_i \leq x'/m_i + s \} ds \\
&\leq \eta_i (x' - x) / m_i, \quad j = 1, 2,
\end{align*}
where the last inequality holds by Lemma~\ref{lem1}.

(ii) Suppose that the two FMS's are different, that is $t_\ast: = \inf \{ t > 0 \colon z^1(t) \neq z^2(t) \} < \infty$.

Without loss of generality we may assume that $t_\ast = 0$. Indeed, otherwise we can consider the time-shifted FMS's $(\zeta^j, z^j)(t_\ast + \cdot)$, $j = 1, 2$. By Lemmas~\ref{lem:fms_bounded} and~\ref{lem:fms_LC}, they start from $z^1(t_\ast) = z^2(t_\ast)  \in  \pos^I$ and $\zeta^1(t_\ast) = \zeta^2(t_\ast)$ with a Lipschitz continuous first projection.

By Lemma~\ref{lem:fms_bounded}, the two NFMS never leave a~compact set $[\dlt,\Dlt]^I \subset \pos^I$. Since the rate functions $\lmb_i(z)$ are Lipschitz continuous in such sets, there exists a~constant $K \in \pos$ such that, for all~$i$ and $s \leq t$,
\[
|S_i(z^1,s,t) - S_i(z^2,s,t)| \leq K t \sup_{s \in [0,t]} \| z^1(s) - z^2(s) \| =: K t \eps(t). 
\]
Then, by Lipschitz continuity of the initial condition, we have, for all $i$ and $t \geq 0$,
\begin{align*}
| z_i^1(t) - z_i^2(t) | \leq& \ L K t \eps(t) + \eta_i \int_0^t \pr \{ S_i(z^1,s,t) - K t \eps(t) \leq B_i \leq S_i(z^1,s,t) + K t \eps(t) \} ds.
\end{align*}
In the last equation, change the variable of integration for $v = S_i(z^1,s,t)$ ({\it cf.} the proof of Lemma~\ref{lem:fms_LC}) and put $M = \sup_{z \in [\dlt,\Dlt]^I} \max_{1 \leq i \leq I} 1 / \lmb_i(z)$. Then
\begin{align*}
| z_i^1(t) - z_i^2(t) | \leq L K t \eps(t) + \eta_i M 2 K t \eps(t) \text{ for all $i$} \quad \text{and} \quad \eps(t) \leq (L + 2 \| \eta \| M) K t \eps(t),\end{align*}
which implies that $\eps(t) = 0$ for small enough $t$, and we arrive at a contradiction with $t_\ast = 0$. \qed

\subsection{Proof of Theorem~\ref{th:unique_fixed_point}} \label{sec:proof_fixed_point}
Before proceeding with the proof of the theorem, we discuss some properties of the functions $g_i(\cdot)$ in the auxiliary Lemmas~\ref{lem:g_increasing} and~\ref{lem:alpha}. Recall that these functions are given by
\[
g_i(x) = \eta_i \ex (B_i \wedge x D_i), \quad x \geq 0.
\]
\begin{lemma} \label{lem:g_increasing}
The function $g_i(\cdot)$ is continuous. Also $g_i(\cdot)$ is strictly increasing in $[0, \alpha_i]$ and constant in $[\alpha_i, \infty)$, where 
\[
\alpha_i := \inf \{ x  \colon g_i(x) = \rho_i \} > 0,
\]
and infimum over the empty set is defined to be $\infty$.
\end{lemma}

\begin{proof}
Continuity of $g_i(\cdot)$ follows by the dominated convergence theorem. 

The situation $\alpha_i = 0$ is not possible since in that case $g_i(x) = \rho_i $ for all $x > 0$ by the definition of $\alpha_i$. But $g_i(\cdot)$ is continuous and $g_i(x) \to g_i(0) = 0$ as $x \to 0$.

If $\alpha_i < \infty$, then, again by the definition of $\alpha_i$ and continuity of $g_i(\cdot)$, we have $g_i(x) = \rho_i$ for all $x \geq \alpha_i$ and $g_i(x) < \rho_i = g_i(\alpha_i)$ for all $x < \alpha_i$.

It is left to check that $g_i(\cdot)$ is strictly increasing in $[0,\alpha_i)$. Assume that $0 \leq x < y < \alpha_i$, but $g_i(x) = g_i(y)$. Then
\begin{align*}
0 = g_i(y)/\eta_i - g_i(x)/\eta_i =& \ex B_i \ind_{\{ B_i \leq x D_i \}} + \ex B_i \ind_{\{ x D_i < B_i \leq y D_i \}} + \ex y D_i \ind_{\{ B_i > y D_i \}} \\
&- \ex B_i \ind_{\{ B_i \leq x D_i \}} - \ex x D_i \ind_{\{ x D_i < B_i \leq y D_i \}} - \ex x D_i \ind_{\{ B_i > y D_i \}}) \\
=& \ex \underbrace{(B_i - x D_i) \ind_{\{ x D_i < B_i \leq y D_i \}}}_{\displaystyle{=:X}} + (y - x) \ex \underbrace{D_i \ind_{\{ B_i > y D_i \}}}_{\displaystyle{=:Y}}, 
\end{align*}
where the r.v.'s $X$ and $Y$ are non-negative, so they must a.s. equal zero. In particular, we have $B_i \leq y D_i$ and $g_i(y)  = \rho_i$, which contradicts the definition of $\alpha_i$ since $y < \alpha_i$. 
\end{proof}

The stabilization points $\alpha_i$ of the functions $g_i(\cdot)$ are related with the r.v.'s $(B_i,D_i)$ in the following way.

\begin{lemma} \label{lem:alpha}
If $\alpha_i < \infty$, then $\inf D_i / B_i = 1 / \alpha_i$. If $\alpha_i = \infty$, then $\inf D_i / B_i = 0$. 
\end{lemma}

\begin{proof}
First assume $\alpha_i < \infty$. Rewrite the relation $g_i(x) = \rho_i$ as ${\ex B_i (1 - (1 \wedge x D_i / B_i)) = 0}$, which, for $x > 0$, is equivalent to $D_i / B_i \geq 1/x$ a.s. Hence $\alpha_i = \inf \{ x > 0 \colon D_i / B_i \geq 1/x \text{ a.s.}\}$ and $1/\alpha_i = \sup \{ y > 0 \colon D_i / B_i \geq y \text{ a.s.} \}$. In the right-hand side of the latter equation we see the definition of $ \inf D_i / B_i$.

Now consider the case $\alpha_i = \infty$. Assume that $\inf D_i / B_i = y>0$, then $D_i / y \geq B_i$ a.s. and $g_i(1/y) = \rho_i$. On the other hand, since $\alpha_i = \infty$, there is no $x > 0$ such that $g_i(x) = \rho_i$. Hence $y = 0$.
\end{proof}

Having established the above properties of the $g_i(\cdot)$'s, we now can prove Theorem~\ref{th:unique_fixed_point} by adapting a technique developed by Borst {\it et al.} \cite[Lemma 5.2]{BEZ09}.

{\it Proof of Theorem~\ref{th:unique_fixed_point}.}
We first  show uniqueness. Let $z^\ast \in \pos^I$ be an invariant NFMS. Recall that $\lmb(z^\ast)$ is the unique optimal solution for the concave optimization problem~\eqref{eq:def_lmb}. The necessary and sufficient conditions for that are given by the Karush-Kuhn-Tucker (KKT) theorem (see e.g.~\cite[Theorem~3.1]{Balder}): there exist $p \in \nneg^J$ and $\tilde{q} \in \nneg^I$ such that
\begin{gather} 
z_i^\ast \, \mathcal{U}'_i(\lmb_i(z^\ast)) = z_i^\ast \sum_{j =1}^J A_{j i} p_j + \tilde{q}_i \quad \text{for all $i$}, \nonumber \\
p_j ( \sum_{i =1}^I A_{j i} \lambda_i(z^\ast) z_i^\ast  - C_j ) = 0 \quad \text{for all $j$},  \label{eq:very_old_kkt}\\
\tilde{q}_i (\lambda_i(z^\ast) - m_i) = 0 \quad \text{for all $i$}. \nonumber
\end{gather}
or equivalently, there exist $p \in \nneg^J$ and $q \in \nneg^I$ ($q_i = \tilde{q}_i/z_i^\ast$) such that
\begin{subequations} \label{eq:old_kkt}
\begin{gather}
\mathcal{U}'_i(\lmb_i(z^\ast)) = \sum_{j =1}^J A_{j i} p_j + q_i \quad \text{for all $i$}, \label{eq:kkt_1} \\
p_j ( \sum_{i =1}^I A_{j i} \Lambda_i(z^\ast) - C_j ) = 0 \quad \text{for all $j$}, \label{eq:kkt_2} \\
q_i (\lambda_i(z^\ast) - m_i) = 0 \quad \text{for all $i$}. \label{eq:kkt_3}
\end{gather}
\end{subequations}
By the assumptions of the theorem and Lemmas~\ref{lem:g_increasing} and~\ref{lem:alpha}, the functions $g_i(\cdot)$ are strictly increasing in the intervals $[0,m_i]$, which implies two things (see also Fig.~\ref{fig:g_i}). First, the fixed point equation~\eqref{eq:fixed_point_equiv} can be rewritten as $\lmb_i(z^\ast) = g_i^{-1}(\Lmb_i(z^\ast))$ for all $i$, and we plug that into~\eqref{eq:kkt_1}. Second, the second multiplier in~\eqref{eq:kkt_3} is zero if and only if $g_i(\lmb_i(z^\ast)) = g_i(m_i)$, and that, by~\eqref{eq:fixed_point_equiv}, is equivalent to $\Lmb_i(z^\ast) = g_i(m_i)$. Hence, $\Lmb(z^\ast)$ satisfies
\begin{subequations} \label{eq:new_kkt}
\begin{gather} 
\mathcal{U}'_i(g_i^{-1}(\Lmb_i(z^\ast)) = \sum_{j=1}^J A_{j i} p_j + q_i \quad \text{for all $i$}, \label{eq:new_kkt_1} \\
p_j ( \sum_{i =1}^I A_{j i} \Lmb_i(z^\ast) - C_j ) = 0 \quad \text{for all $j$}, \label{eq:new_kkt_2} \\
q_i (\Lmb_i(z^\ast) - g_i(m_i)) = 0 \quad \text{for all $i$} \label{eq:new_kkt_3}.
\end{gather}
\end{subequations}
Now note that the last three equations form the KKT conditions for another optimization problem. Indeed, take functions $\tilde{g}_i(\cdot)$ that are continuous and strictly increasing in $\nneg$ and coincide with $g_i(\cdot)$ in $[0,m_i]$ (and hence, the inverse functions $\tilde{g}_i^{-1}(\cdot)$ and $g_i^{-1}(\cdot)$ coincide in $[0,g_i(m_i)]$). Also take functions $G_i(\cdot)$ such that $G_i'(\cdot) = \mathcal{U}'_i (\tilde{g}_i^{-1}(\cdot))$ in $\pos$. Then~\eqref{eq:new_kkt} gives necessary and sufficient conditions for $\Lmb(z^\ast)$ to solve~\eqref{eq:th_unique_fixed_point}.
Since the functions $\mathcal{U}_i(\cdot)$ are strictly concave, their derivatives $\mathcal{U}'_i(\cdot)$ are strictly decreasing. Then $G'_i(\cdot)$ are strictly decreasing and, equivalently, $G_i(\cdot)$ are strictly concave, which implies that $\Lmb(z^\ast) = \Lmb^\ast$ is actually the unique solution to~\eqref{eq:th_unique_fixed_point} and does not depend on $z^\ast$. Then we invert the $g_i(\cdot)$'s in the fixed point equation~\eqref{eq:fixed_point_equiv}, which implies that the fixed point $z_i^\ast = \Lmb_i^\ast / g_i^{-1}(\Lmb_i^\ast)$ is unique because $\Lmb^\ast$ is unique.

The existence result follows similarly. There exists a unique optimal solution $\Lmb^\ast$ to~\eqref{eq:th_unique_fixed_point} and it satisfies the KKT conditions~\eqref{eq:new_kkt}. Put $\lmb_i^\ast = g_i^{-1}(\Lmb_i^\ast)$ for all~$i$. Then $\lmb^\ast$ and $\Lmb^\ast$ satisfy the KKT conditions~\eqref{eq:old_kkt} and~\eqref{eq:very_old_kkt} , i.e., for the vector $z^\ast$ with $z_i^\ast := \Lmb_i^\ast / \lmb_i^\ast$, we have $\lmb^\ast = \lmb(z^\ast)$ and $\Lmb^\ast = \Lmb(z^\ast)$. Plugging the last two relations into the definition of $\lmb^\ast$, we get the fixed point equation. \qed

\begin{figure}[!htb] 
\centering
\includegraphics[scale=0.65]{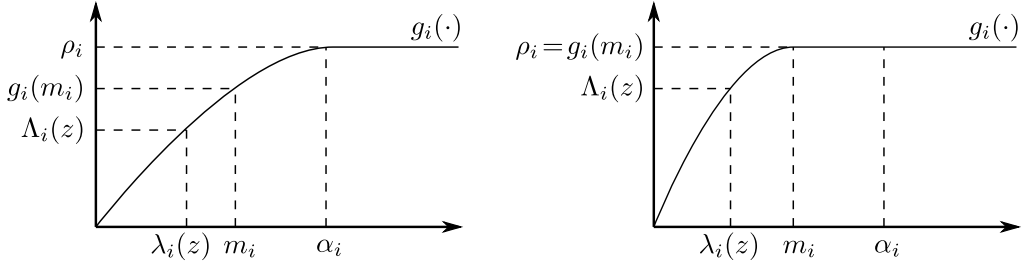}
\caption{Graph of the function $g_i(\cdot)$ in the two possible cases: \\ when $m_i \leq \alpha_i$ (left) and when $m_i > \alpha_i$ (right); $z$ is a invariant NFMS. }
\label{fig:g_i}
\end{figure}

\subsection{Proof of Theorem~\ref{th:single_link}}
The fixed point equation~\eqref{eq:fixed_point_equiv} and the monotonicity of the functions $g_i(\cdot)$ imply that the bandwidth class~$i$ gets in an equilibrium is at most $g_i(m_i)$. Therefore, we refer to the scenarios $\sum_{i=1}^I g_i(m_i) < C$, $\sum_{i=1}^I g_i(m_i) = C$ and $\sum_{i=1}^I g_i(m_i) > C$ as underloaded, critically loaded and overloaded, respectively. Below we calculate the invariant NFMS's in the three cases.

Summing up~\eqref{eq:fixed_point}, the KKT conditions~\eqref{eq:old_kkt} for~\eqref{eq:def_lmb} and the capacity and rate constraints, a $z \in \pos^I$ is an invariant NFMS if and only if there exist $p \in \nneg$ and $q \in \nneg^I$ such that (we omit the argument of the rates $\lmb_i(z)$ and bandwidth allocations $\Lmb_i(z)$)
\begin{subequations} \label{eq:single_link}
\begin{gather}
\Lmb_i = g_i(\lmb_i) \label{eq:single_link_a} \quad \text{for all~$i$}, \\
\kappa_i / \lmb_i = p + q_i \label{eq:single_link_b} \quad \text{for all~$i$}, \\
p (\sum_{i=1}^I \Lmb_i - C) = 0 \label{eq:single_link_c} \\
q_i(\lmb_i - m_i) = 0 \label{eq:single_link_d} \quad \text{for all~$i$}, \\
\sum_{i=1}^I \Lmb_i \leq C \label{eq:single_link_e}, \\
\lmb_i \leq m_i \label{eq:single_link_f}  \quad \text{for all~$i$}.
\end{gather}
\end{subequations}

\paragraph{Underload} In this case, there is no interaction between the classes, they do not compete but all get the maximum rate allowed. Indeed,~\eqref{eq:single_link_c} and~\eqref{eq:single_link_b} imply that $p = 0$ and all $q_i > 0$. Then, by~\eqref{eq:single_link_d} and~\eqref{eq:single_link_a}, all $\lmb_i = m_i$ and all $\Lmb_i = g_i(m_i)$. Hence, there is a~unique invariant NFMS given by 
\[
z_i = g_i(m_i)/m_i \quad \text{for all~$i$}.
\]

\paragraph{Critical load} First note that 
\begin{equation} \label{eq6}
\Lmb_i = g_i(m_i) \quad \text{for all $i$}.
\end{equation}
Indeed, there are two possibilities: either $p = 0$ $\stackrel{\eqref{eq:single_link_b}}{\Rightarrow}$ all $q_i > 0$ $\stackrel{\eqref{eq:single_link_d}}{\Rightarrow}$ all $\lmb_i = m_i$ $\stackrel{\eqref{eq:single_link_a}}{\Rightarrow}$ \eqref{eq6}, or $p>0$ $\stackrel{\eqref{eq:single_link_c}}{\Rightarrow}$ $\sum_{i=1}^I \Lmb_i = C$ $\Rightarrow$ \eqref{eq6}, where the last implication is due to $\Lmb_i \leq g_i(m_i)$ and $\sum_{i = 1}^I g_i(m_i) = C$.

Recall from Lemma~\ref{lem:alpha} that
\[
\alpha_i := \inf \{ x \colon g_i(x) = \rho_i \} = 1 / \inf(D_i/B_i).
\]
By~\eqref{eq6}, the relations \eqref{eq:single_link_a} and \eqref{eq:single_link_f} are equivalent to $m_i \wedge \alpha_i \leq \lmb_i \leq m_i$ ( see Fig.~\ref{fig:g_i}). Hence, \eqref{eq:single_link}~reduces to
\begin{subequations} \label{eq:critical_load}
\begin{gather}
\kappa_i / \lmb_i = p + q_i \label{eq:critical_load_a}, \\
q_i(\lmb_i - m_i) = 0 \label{eq:critical_load_b},\\
m_i \wedge \alpha_i \leq \lmb_i \leq m_i \label{eq:critical_load_c}.
\end{gather}
\end{subequations}
Let
\[
\cli_{\text{crit}}: = \{ i \colon m_i \leq \alpha_i \}.
\]
For $i \in \cli_{\text{crit}}$, by~\eqref{eq:critical_load_c}, we have $\lmb_i = m_i$ and $z_i = g_i(m_i) / m_i$. Then~\eqref{eq:critical_load_b} is satisfied, and~\eqref{eq:critical_load_a} implies that $p \leq \kappa_i / m_i$.

Now divide $\{1, \ldots, I\} \setminus \cli_{\text{crit}}$ into two subsets $\cli_1 \cap \cli_2 = \varnothing$. For $i \in \cli_1$, put $\lmb_i = m_i$, then (as for $i \in \cli_{\text{crit}}$) $z_i = g_i(m_i) / m_i$, \eqref{eq:critical_load_b} is satisfied, and~\eqref{eq:critical_load_a} implies that $p \leq \kappa_i / m_i$. For $i \in \cli_2$, assume $\alpha_i \leq \lmb_i < m_i$. Then $q_i = 0$ by \eqref{eq:critical_load_b}, $\kappa_i / \lmb_i = p$ by \eqref{eq:critical_load_a}, and $\kappa_i / m_i < p \leq \kappa_i / m_i$. Also $z = g_i(m_i) p / \kappa_i$.

Summing up everything said above, the set of invariant NFMS's is given by
\begin{align*}
S_z := \bigcup_{\cli \supseteq \cli_{\text{crit}}} \{ z \colon &z_i =g_i(m_i)/m_i \text{ for } i \in \cli \ \text{ and } \
z_i = g_i(m_i)p / \kappa_i \text{ for } i \notin \cli, \\
&\text{where } p \in ( \max_{i \notin \cli} \kappa_i / m_i, \min_{i \in \cli} \kappa_i / m_i \wedge \min_{i \notin \cli} \kappa_i / \alpha_i ] \}.
\end{align*}
Equivalent descriptions of $S_z$ are
\begin{align*}
S_z &= \{ z \colon z_i = g_i(m_i) / m_i \text{ if } p \leq \kappa_i / m_i \ \text{ and } \ z_i = g_i(m_i)p / \kappa_i \text{ if } p > \kappa_i / m_i, \quad p \in S_p\},\\
S_p &:= (0, \min_{i \in \cli_{\text{crit}}} \kappa_i / m_i \wedge \min_{i \notin \cli_{\text{crit}}} \kappa_i / \alpha_i ] = (0, \min_{1 \leq i \leq I} \kappa_i / (m_i \wedge \alpha_i) ],
\end{align*}
and
\begin{align*}
S_z &= \{ z \colon z_i = g_i(m_i) / (m_i \wedge \kappa_i x), \quad x \in S_x \},\\
S_x &:= [\max_{1 \leq i \leq I} (m_i \wedge \alpha_i) / \kappa_i, \infty).
\end{align*}
We now apply the last formula in a couple of simple examples.

{\it Example 1.} If $m_i \leq \alpha_i$ for all $i$, then $S_x = [ \max_{1 \leq i \leq I} m_i / \kappa_i, \infty )$, and $\kappa_i x \geq m_i$ for all $x \in S_x$ and all~$i$. Hence, there is a unique invariant NFMS given by $z_i = g_i(m_i) / m_i$ for all $i$, which agrees with Theorem~\eqref{th:unique_fixed_point}.

{\it Example 2.} If $m_1 > \alpha_1$, $m_i \leq \alpha_i$ for $i \neq 1$ and $\alpha_1 / \kappa_1 \geq \max_{i \neq 1} m_i / \kappa_i$, then, for any $\lambda_1 \in [\alpha_1, m_1]$, $z = (g_1(m_1) / \lambda_1, g_2(m_2) / m_2, \ldots, g_I(m_I) / m_I)$ is an invariant NFMS.

\paragraph{Overload} In this situation, by the capacity constraint~\eqref{eq:single_link_e}, at least one class of flows does not receive the maximum service, i.e. at least one $\Lambda_i < g_i(m_i)$. We first find out which classes get the maximum service and which do not, and then calculate the unique invariant NFMS.

{\it Who gets the maximum service.} Since at least one $\Lmb_i < g_i(m_i)$, at least one $\lmb_i < m_i \wedge \alpha_i$ (see Fig.~\ref{fig:g_i}). Then \eqref{eq:single_link_d}, \eqref{eq:single_link_b} and~\eqref{eq:single_link_c} imply that at least one $q_i  = 0$, $p > 0$ and $\sum_{i = 1}^I \Lmb_i = C$. At~this~point, we can equivalently rewrite~\eqref{eq:single_link} as follows: there exist $x > 0$ and $\eps \in \nneg^I$ such that (the functions $\tilde{g}_i(\cdot)$ are introduced below)
\begin{subequations} \label{eq:overload}
\begin{gather}
\Lmb_i = g_i(\lmb_i) \Leftrightarrow \Lmb_i = \tilde{g}_i(\lmb_i), \label{eq:overload_a}\\
\sum_{i = 1}^I g_i(\lmb_i)  = C \Leftrightarrow \sum_{i = 1}^I \tilde{g}_i(\lmb_i)  = C \label{eq:overload_b}\\
\lmb_i = \kappa_i( x - \eps_i) \label{eq:overload_c}, \\
\eps_i(\lmb_i - m_i) = 0 \label{eq:overload_d}, \\
\lmb_i \leq m_i \label{eq:overload_e}.
\end{gather}
\end{subequations}

For all $i$ and $x \geq 0$, put
\[
\tilde{g}_i(x) := g_i(m_i \wedge x).
\]
By the rate constraints~\eqref{eq:overload_e},  in \eqref{eq:overload}, we can equivalently replace $g_i(\cdot)$ by $\tilde{g}_i(\cdot)$.

If $\lmb_i < m_i$, then, by \eqref{eq:overload_d} and \eqref{eq:overload_c}, $\eps_i = 0$ and $\lmb_i = \kappa_i x$, and hence
\begin{equation} \label{eq7}
\tilde{g}_i(\lmb_i) = \tilde{g}_i(\kappa_i x).
\end{equation}
If $\lmb_i = m_i$, then, by~\eqref{eq:overload_c}, $\kappa_i x \geq m_i$ and $\tilde{g}_i(\kappa_i x) = g_i(m_i)$, and, again, \eqref{eq7} holds.

Plugging~\eqref{eq7} into~\eqref{eq:overload_b}, we get
\begin{equation} \label{eq8}
\sum_{i = 1}^I \tilde{g}_i(\kappa_i x) = C.
\end{equation}
The function $\tilde{g}(x) := \sum_{i = 1}^I \tilde{g}_i(\kappa_i x)$ is continuous everywhere, strictly increasing in the interval 
\[
0 \leq x \leq \max_{1 \leq i \leq I} (m_i \wedge \alpha_i) / \kappa_i =: x_0
\]
and constant for $x \geq x_0$, and also $\tilde{g}(0) = 0$ and $\tilde{g}(x_0) = \sum_{i=1}^I g_i(m_i) > C$, which implies that there exists a unique $x$ solving~\eqref{eq8} and $x \in (0, x_0)$.

By~\eqref{eq:overload_a} and~\eqref{eq7}, $\Lmb_i = \tilde{g}_i(\kappa_i x)$. Then (see Fig.~\ref{fig:tilde_g_i}) $\Lmb_i = g_i(m_i)$ if $(m_i \wedge \alpha_i) / \kappa_i \leq x$ and $\Lmb_i < g_i(m_i)$ if $(m_i \wedge \alpha_i) / \kappa_i > x$. Hence, the set of classes that get the maximum service is
\begin{equation} \label{eq:max_service}
\cli_{\text{over}} := \{ i \colon (m_i \wedge \alpha_i) / \kappa_i \leq x \}.
\end{equation}

{\it Invariant NFMS.} For $i \notin \cli_{\text{over}}$, $\Lmb_i = g_i(\lmb_i) < g_i(m_i)$, which implies that (see Fig.~\ref{fig:g_i}) $\lmb_i < m_i \wedge \alpha_i$. Then, by \eqref{eq:overload_d} and~\eqref{eq:overload_c}, $\eps_i = 0$ and $\lmb_i = \kappa_i x$ (meeting the rate constraint~\eqref{eq:overload_e}), and $z_i = \Lmb_i / \lmb_i = g_i(\kappa_i x) / (\kappa_i x)$.

For $i \in \cli_{\text{over}}$, consider the two possible cases: $\kappa_i x < m_i$ and $\kappa_i x \geq m_i$. If $\kappa_i x < m_i$,  then, by \eqref{eq:overload_c} and~\eqref{eq:overload_d}, $\lmb_i \leq \kappa_i x < m_i$ and $\eps_i  = 0$, and, again by \eqref{eq:overload_c}, $\lmb_i = \kappa_i x$. If $\kappa_i x \geq m_i$, then $\lmb_i = m_i$ because otherwise we would arrive at a contradiction: $\lmb_i < m_i$ $\stackrel{\eqref{eq:overload_d}}{\Rightarrow}$ $\eps_i = 0$ $\stackrel{\eqref{eq:overload_c}}{\Rightarrow}$ $\lambda_i = \kappa_i x \geq m_i$. Hence, for $i \in \cli_{\text{over}}$, $\lmb_i = m_i \wedge \kappa_i x$ and $z_i = \Lmb_i / \lmb_i = g_i(m_i) / (m_i \wedge \kappa_i x)$.

Summing up, the unique invariant NFMS is given by
\[
z_i = g_i(m_i) / (m_i \wedge \kappa_i x) \text{ for } i \in \cli_{\text{over}} \ \text{ and }\  z_i = g_i(\kappa_i x) / (\kappa_i x) \text{ for }  i \notin \cli_{\text{over}},
\]
where $x$ is the unique solution to~\eqref{eq8} and $\cli_{\text{over}}$ is defined by~\eqref{eq:max_service}.

\begin{figure}[!htb] 
\centering
\includegraphics[scale=0.60]{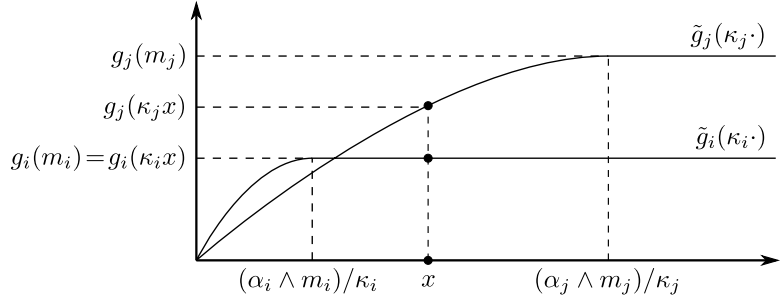}
\caption{Graphs of the functions $\tilde{g}_i(\cdot)$; $x$ is the unique solution to~\eqref{eq8}. }
\label{fig:tilde_g_i}
\end{figure}

\section{Proof of Theorem~\ref{th:fluid_limits}} \label{sec:proof_fluid_limits}
To prove $\mathbf{C}$-tightness (that is, tightness with all weak limits being a.s. continuous) of $\{ \cloz^r(\cdot) \}_{r \in \clr}$, we check standard conditions (see e.g. \cite{EthierKurtz}) of compact containment (Section ~\ref{sec:compact_containment}) and oscillation control (Section~\ref{sec:oscillation_control}). In Section~\ref{sec:fluid_limits_fms}, we check that fluid limits satisfy the fluid model equation~\eqref{eq:mvfms}.

To establish these main steps of the proof, we develop a number of auxiliary results. Section~\ref{sec:load} contains a law of large numbers result for the load process. Section~\ref{sec:asymptotic_regularity} proves that, for large~$r$, $\cloz^r(\cdot)$ puts arbitrarily small mass to thin horizontal and vertical strips, which in particular implies that fluid limits have both projections free of atoms. In Section~\ref{sec:fluid_limits_bounded_from_zero}, fluid limits are shown to be coordinate-wise bounded away from zero outside $t = 0$.

\subsection{Load process} \label{sec:load}
Introduce the measure valued load processes and their scaled versions: for all $r$, $i$ and $t \geq s \geq 0$,
\begin{equation*}
\begin{alignedat}{2}
\cll_i^r(t) &:= \sum_{k=1}^{E_i^r(t)} \dlt_{(B_{ik}^r, D_{ik}^r)},& \quad \quad \quad \clol_i^r(t) &:= \cll_i^r(t) / r,\\
\cll_i^r(s,t) &:= \cll_i^r(t) - \cll_i^r(s),& \quad \quad \quad \clol_i^r(s,t) &:= \clol_i^r(t) - \clol_i^r(s).
\end{alignedat}
\end{equation*}

The following property is useful when proving other results of the section. Only minor adjustments in the proof of~\cite[Theorem~5.1]{GW09} are needed to establish it.
\begin{lemma} \label{lem:load}
By \textup{(A.3) {\it and} (A.4)}, as $r \to \infty$, $(\clol^r(\cdot), \la \chi_1, \clol^r(\cdot) \ra, \la \chi_2, \clol^r(\cdot) \ra) \Rightarrow (\eta(\cdot) \ast \theta, \rho(\cdot), \sgm(\cdot))$, where 
$\chi_1(x_1,x_2) := x_1$, $\chi_2(x_1, x_2) := x_2$, and $\eta(t) := t \eta$, $\rho(t) := t \rho$, $\sgm(t) := t \sgm$.
\end{lemma}

\subsection{Compact containment} \label{sec:compact_containment}
The property we prove here, together with the oscillation control result that follows in Section~\ref{sec:oscillation_control}, implies tightness of the scaled state descriptors.
\begin{lemma} \label{lem:compact_containment}
By \textup{(A.3)--(A.5)}, for any $T > 0$ and $\eps > 0$, there exists a compact set $K \subset \msr^I$ such that 
\[
\liminf_{r \to \infty} \pr^r \{ \cloz^r(t) \in K \text{ for all } t \in [0,T] \} \geq 1 - \eps.
\]
\end{lemma}

\begin{proof}
Fix $T$ and $\eps$. It suffices to show that, for each $i$, there exist a~compact set $K_i \subset \msr$ such that 
\begin{equation} \label{eq9}
\liminf_{r \to \infty} \pr^r \{ \cloz_i^r(t) \in K_i \text{ for all } t \in [0,T] \} \geq 1 - \eps / I.
\end{equation}

We use the following criterion (see e.g. \cite[Theorem 15.7.5]{Kallenberg}).
\begin{proposition} \label{pr:rel_comp_criterion}
A set $\clm \subset \msr$ is relatively compact if and only if $\sup_{\xi \in \clm} \xi(\nneg^2) < \infty$ and $\sup_{\xi \in \clm} \xi(\nneg^2 \setminus [0,n]^2) \to 0$ as $n \to \infty$.
\end{proposition}

Note that
\begin{equation} \label{eq10}
\cloz_i^r (t)(\nneg^2)  = \oz_i^r (t) \leq \oz_i^r (0) + \ove_i^r (T) = \cloz_i^r (0)(\nneg^2) + \clol_i^r (T)(\nneg^2).
\end{equation}
Also note that, if the residual size (patience time) of a flow at time~$t$ exceeds~$n$, then its initial size (patience time), must have exceeded~$n$, too, which implies the following bound:
\begin{equation} \label{eq11}
\cloz_i^r (t)(\nneg^2 \setminus [0,n]^2) \leq \cloz_i^r (0)(\nneg^2 \setminus [0,n]^2) + \clol_i^r (T)(\nneg^2 \setminus [0,n]^2).
\end{equation}
The sequence $\{ \cloz_i^r (0) + \clol_i^r (T) \}_{r \in \clr}$ converges and hence in tight, i.e. there exists a compact set $K'_i \subset \msr$ such that
\begin{equation} \label{eq12}
\inf_{r \in \clr} \pr^r \{ \cloz_i^r (0) + \clol_i^r (T) \in K'_i \} \geq 1 - \eps/I.
\end{equation}
Put 
\begin{align*}
K_i'' := \{ \xi \in \msr \colon \text{for some } \xi' \in K'_i, \ &\xi(\nneg^2) \leq \xi'(\nneg^2) \text{ and}\\
&\xi(\nneg^2 \setminus [0,n]^2) \leq \xi'(\nneg^2 \setminus [0,n]^2), \ n \in \nat \}.
\end{align*}
Then the criterion of relative compactness for $K_i''$ follows from that for $K'_i$, and~\eqref{eq10}--\eqref{eq12} imply~\eqref{eq9} with $K_i$ taken as the closure of $K_i''$.
\end{proof}

\subsection{Asymptotic regularity} \label{sec:asymptotic_regularity}
This section contains three Lemmas. Lemmas~\ref{lem:as_reg_z0} and~\ref{lem:as_reg_znew} prove that neither initial nor newly arriving flows concentrate along horizontal and vertical lines. These two results are combined in Lemma~\ref{lem:as_reg_zr} that implies the oscillation control result of the next section, and also is useful when deriving the limiting equations for the state descriptors in Section~\ref{sec:fluid_limits_fms}.

Recall from Section~\ref{sec:stat_distributions} that, for $b \geq a \geq 0$,
\[
H_a^b = \nneg \times [a,b], \quad V_a^b = [a,b] \times \nneg,
\]
and introduce similar notations
\[
H_a^\infty := \nneg \times [a, \infty), \quad V_a^\infty := [a, \infty) \times \nneg.
\]

\begin{lemma} \label{lem:as_reg_z0}
By \textup{(A.5) {\it and} (A.6)}, for any $\dlt > 0$ and $\eps > 0$, there exists an $a > 0$ such that 
\[
\liminf_{r \to \infty} \pr^r \{ \sup_{x \in \nneg} \| \cloz^r(0)(H_x^{x+a}) \| \vee \| \cloz^r(0)(V_x^{x+a}) \| \leq \dlt \} \geq 1 - \eps.
\]
\end{lemma}

{\it Proof.}
Fix $\dlt$ and $\eps$. Since, for any $\xi \in \msr^I$ and $a > 0$,  
\[
\sup_{x \in \nneg}  \| \xi(H_x^{x+a}) \| \vee \| \xi(V_x^{x+a}) \| \leq 2 \sup_{n \in \nat}  \| \xi(H_{(n-1) a}^{n a}) \| \vee \| \xi(V_{(n-1) a}^{n a}) \|,
\]
it suffices to find an $a$ such that 
\[
\liminf_{r \to \infty} \pr^r \{ \cloz^r(0) \in \clm_a \} \geq 1 - \eps,
\]
where $\clm_a := \{ \xi \in \msr^I \colon \sup_{n \in \nat}  \| \xi(H_{(n-1) a}^{n a}) \| \vee \| \xi(V_{(n-1) a}^{n a}) \| < \dlt / 2  \}$.

The set $\clm_a$ is open because $\xi^k \wto \xi \in \clm_a$  implies that $\xi^k \in \clm_a$ for $k$ large enough. Indeed, pick an $N \in \nat$ such that $\| \xi(H_{Na}^\infty) \| \vee \| \xi(V_{Na}^\infty) \| < \dlt / 2$. Then, by the Portmanteau theorem,
\begin{align*}
&\limsup_{k \to \infty} \sup_{n \in \nat} \| \xi^k (H_{(n-1) a}^{n a}) \| \vee \| \xi^k (V_{(n-1) a}^{n a}) \| \\
\leq &\limsup_{k \to \infty} \max_{1 \leq n \leq N} \| \xi^k (H_{(n-1) a}^{n a}) \| \vee \| \xi^k (V_{(n-1) a}^{n a}) \| \vee 
\| \xi^k (H_{N a}^\infty) \| \vee \| \xi^k (V_{N a}^\infty) \| \\
\leq & \max_{1 \leq n \leq N} \| \xi (H_{(n-1) a}^{n a}) \| \vee \| \xi (V_{(n-1) a}^{n a}) \| \vee 
\| \xi (H_{N a}^\infty) \| \vee \| \xi (V_{N a}^\infty) \| < \dlt/2.
\end{align*}
By \textup{(A.6)} and \cite[Lemma A.1]{GW09}, there exists an~$a$ such that $\pr \{ \zeta^0 \in \clm_a \} \geq 1 - \eps$. Then, again by the Portmanteau theorem,
\begin{flalign*}
& &\liminf_{r \to \infty} \pr^r \{ \cloz^r(0) \in \clm_a \} \geq \pr \{ \zeta^0 \in \clm_a \} \geq 1 - \eps. & & \qed
\end{flalign*}

Besides being used in the proof of the fluid limit theorem, the following result is also used when establishing convergence of the stationary distributions of the scaled state descriptors, see Section~\ref{sec:stat_distributions}.
\begin{lemma} \label{lem:as_reg_znew}
By \textup{(A.3) {\it and} (A.4)}, for any $T > 0$, $\dlt > 0$ and $\eps > 0$, there exists an $a > 0$ such that 
\[
\liminf_{r \to \infty} \pr^r \{ \underbrace{ \sup_{t \in [0,T]} \sup_{x \in \nneg} \| \cloznew(t)(H_x^{x+a}) \| \vee \| \cloznew(t)(V_x^{x+a}) \| \leq \dlt }_{\displaystyle{=: \Omega_\ast^r}} \} \geq 1 - \eps.
\]
\end{lemma}

\begin{proof}
Fix $T$, $\dlt$ and $\eps$. We first construct auxiliary events $\Omega_0^r$ such that $\liminf_{r \to \infty} \pr^r \{ \Omega_0^r \} \geq 1 - \eps$, and then show that $\Omega_\ast^r \supseteq \Omega_0^r$ for all~$r$, which implies the theorem.

{\it Definition of $\Omega_0^r$.} By Lemma~\ref{lem:compact_containment}, there exists a compact set $K \subset \msr^I$ such that
\[
\liminf_{r \to \infty} \pr^r \{ \underbrace{\cloz^r(t) \in K \text{ for all } t \in [0,T] }_{\displaystyle{=: \Omega_1^r}} \} \geq 1 - \eps,
\]
and by Proposition~\ref{pr:rel_comp_criterion}, $M := \sup_{\xi \in K} \| \xi(\nneg^2) \| < \infty$ and $\sup_{\xi \in K} \| \xi(\nneg^2) \setminus [0,L]^2 \| \leq \dlt / 4$ for a~large~enough $L$.

For each~$i$, the rate function $\lmb_i(\cdot)$ is positive on $\{ z \in \nneg^I \colon z_i > 0 \}$ and, by Lemma~\ref{lem:Lmb_continuous}, it is continuous there. Hence,
\begin{equation} \label{eq17}
\lmb_\ast := \min_{1 \leq i \leq I} \inf \{ \lmb_i(z) \colon z_i \geq \dlt  / 4, \, \| z \| \leq M \} > 0.
\end{equation}
Put
\[
\gamma : = \frac{\dlt}{72 \| \eta \|} \wedge T \quad \text{and} \quad a:= \frac{\gamma(\lmb_\ast \wedge 1)}{3}.
\]
Also pick an $N$ large enough so that
\[
Na > L + (\|m\| \vee 1) T.
\]
For $m,n \in \nat$, define the sets
\begin{align*}
I_{m,n} &:= [(m-1)a, ma) \times [(n-1)a, na), \\
I^{m,n} &:= [(m-2)^+ a, (m+1)a) \times [(n-2)^+a, (n+1)a),
\end{align*}
and pick functions $g_{m,n} \in \mathbf{C}_{\nneg^2 \to [0,1]}$ such that
\[
\ind_{I_{m,n}}(\cdot) \leq g_{m,n}(\cdot) \leq \ind_{I^{m,n}}(\cdot).
\]
Since $\theta$ is a vector of probability measures,
\begin{equation} \label{eq18}
\sum_{m,n \in \nat} \| \la g_{m,n}, \theta \ra \| \leq \| \sum_{m,n \in \nat} \theta(I^{m,n}) \| \leq 9.
\end{equation}
By Lemma~\ref{lem:load} and the continuous mapping theorem, for all $m,n \in \nat$, $\la g_{m,n}, \clol^r(\cdot) \ra \Rightarrow \eta(\cdot)\la g_{m,n}, \theta \ra$ as $r  \to \infty$. Since the limits are deterministic, we have convergence in probability. Since the limits are continuous, we have uniform convergence on compact sets. Hence, 
\[
\lim_{r \to \infty} \pr^r \{ \underbrace{ \max_{1 \leq m,n \leq N} \sup_{t \in [0,T]} \| \la g_{m,n}, \clol^r(t) \ra - t \eta \ast \la g_{m,n}, \theta \ra \| 
\leq \dlt / (16 N^2) }_{\displaystyle{=: \Omega_2^r}} \} = 1.
\]
Similarly, by \textup{(A.3)}, 
\[
\lim_{r \to \infty} \pr^r \{ \underbrace{ \sup_{t \in [0,T]} \| \ove^r(t) - t \eta \| 
\leq \dlt / 16 }_{\displaystyle{=: \Omega_3^r}} \} = 1.
\]
For all $r$, put 
\[
\Omega_0^r := \Omega_1^r \cap \Omega_2^r \cap \Omega_3^r,
\]
then $\liminf_{r \to \infty} \pr^r \{ \Omega_0^r \} \geq 1 - \eps$, and it is left to show that $\Omega_0^r \subseteq \Omega_\ast^r$.

{\it Proof of $\Omega_0^r \subseteq \Omega_\ast^r$.} Fix $r \in \clr$, $t \in [0,T]$, $x \in \nneg$ and $i$. Also fix an outcome $\omega \in \Omega_0^r$. All random objects in the rest of the proof will be evaluated at this~$\omega$. We have to check that
\begin{subequations}
\begin{align}
\cloznew_i(t)(H_x^{x+a}) \leq \delta, \label{eq13} \\
\cloznew_i(t)(V_x^{x+a}) \leq \delta. \label{eq14}
\end{align}
\end{subequations}
We will show~\eqref{eq13}, \eqref{eq14} follows similarly.

Define the random time $\tau := \sup \{ s \leq t \colon \oznew_i(s) < \dlt / 4 \}$ (supremum over the empty set equals $0$ by convention). Although in general  $\tau$ is not a continuity point for $\oznew_i(\cdot)$, we still can estimate $\oznew_i(\tau)$:
\begin{equation} \label{eq15}
\oznew_i(\tau) \leq \dlt / 2.
\end{equation}
Indeed, if $\tau = 0$, then $\oznew_i(\tau) = 0$, and~\eqref{eq15} holds. If $\tau > 0$, pick a~$\tau' \in [(\tau - \gamma)^+, \tau]$ such that $\oznew_i(\tau') < \dlt / 4$. Then, by the definition of $\Omega_3^r$,
\[
\oznew_i(\tau) \leq \oznew_i(\tau') + (\ove_i^r(\tau) - \ove_i^r(\tau')) \leq \dlt / 4 + \eta_i(\tau - \tau') + \dlt / 8 \leq \| \eta \| \gamma + 3 \dlt / 8,
\]
and~\eqref{eq15} holds by the choice of $\gamma$.

Now, if $\tau = t$, then~\eqref{eq15} implies~\eqref{eq13}, and the proof is finished. Assume that $\tau < t$. Then, by the choice of $L$ and~\eqref{eq15},
\begin{gather*}
\cloznew_i(t)(H_x^{x+a}) \leq \cloznew_i(t)(H_x^{x+a} \cap [0,L]^2) + \dlt / 4 \\
\leq \underbrace{\oznew_i(\tau)}_{\displaystyle{\leq \dlt /2}} + \frac{1}{r} \sum_{E_i^r(\tau)+1}^{E_i^r(t)} \underbrace{ \ind_{H_x^{x+a} \cap [0,L]^2}(B_{ik}^r - S_i(\oz^r,U_{ik}^r,t), D_{ik}^r - (t - U_{ik}^r) }_{\displaystyle{=: s_k}} + \dlt / 4
\end{gather*}
and in order to have~\eqref{eq13}, it suffices to show that
\begin{equation} \label{eq19}
\Sgm := \frac{1}{r} \sum_{E_i^r(\tau)+1}^{E_i^r(t)} s_k = \sum_{m,n \in \nat} \underbrace{\frac{1}{r} \sum_{E_i^r(\tau)+1}^{E_i^r(t)} s_k \ind_{I_{m,n}}(B_{ik}^r, D_{ik}^r)}_{\displaystyle{=: \Sgm_{m,n}}} \leq \dlt / 4.
\end{equation}
First note that 
\begin{equation} \label{eq20}
\Sgm_{m,n} = 0 \quad \text{if} \quad m > N \text{ or } n > N.
\end{equation}
Indeed, consider a flow on route $i$ that arrived at $U_{ik}^r \in (\tau,t]$ with $(B_{ik}^r, D_{ik}^r) \in I_{m,n}$. If $m > N$, then $B_{ik}^r > L + \|m\|T$ by the choice of $N$, $B_{ik}^r - S_i(\oz^r,U_{ik}^r,t) > L$ by the rate constraints, and $s_k = 0$. If $n > N$, then $D_{ik}^r > L + T$ by the choice of $N$, $D_{ik}^r - (t - U_{ik}^r) > L$ and again $s_k = 0$.

Now we estimate $\Sgm_{m,n}$ for $1 \leq m,n \leq N$. Fix $m$, $n$. Consider two flows $k < l$ such that $U_{ik}^r, U_{il}^r \in (\tau,t]$ and $(B_{ik}^r,D_{ik}^r), (B_{il}^r,D_{il}^r) \in I_{m,n}$. In $(\tau, t]$, $\oz_i^r(\cdot) \geq \oznew_i(\cdot) \geq \eps / 4$ and $\| \oz^r(\cdot) \| \leq M$, and then \eqref{eq17} implies that
\[
\inf_{s \in (\tau,t]} \lmb_i(\oz^r(s)) \geq \lmb_\ast.
\]
If $U_{il}^r - U_{ik}^r \geq \gamma$, then
\begin{align*}
(B_{il}^r - S_i(\oz^r,U_{il}^r,t)) - (B_{ik}^r - S_i(\oz^r,U_{ik}^r,t)) \geq \overbrace{\gamma \lmb_\ast }^{\displaystyle{\geq 3a}} - \overbrace{(B_{ik}^r - B_{il}^r)}^{\displaystyle{\leq a}} \geq 2a, \\
(D_{il}^r - (t - U_{il}^r) - (D_{ik}^r - (t - U_{ik}^r)) \geq \underbrace{\gamma}_{\displaystyle{\geq 3a}} - \underbrace{(D_{ik}^r - D_{il}^r)}_{\displaystyle{\leq a}} \geq 2a,
\end{align*}
and hence at most one of $s_k$ and $s_l$ is non-zero. This implies that all arrivals to route~$i$ during $(\tau,t]$ that correspond to non-zero summands in $\Sgm_{m,n}$ occur actually during a smaller interval $(t_{m,n}, t_{m,n}+\gamma] \subseteq (\tau,t]$. Then, by the definition of $\Omega_2^r$,
\begin{align*}
\Sgm_{m,n} &\leq \frac{1}{r} \sum_{k = E_i^r(t_{m,n})+1}^{E_i^r(t_{m,n}+\gamma)} \ind_{I_{m,n}}(B_{ik}^r, D_{ik}^r) \leq \sup_{s \in [0, T - \gamma]} \frac{1}{r} \sum_{k = E_i^r(s+1)}^{E_i^r(s+\gamma)} g_{m,n}(B_{ik}^r, D_{ik}^r) \\
&= \sup_{s \in [0, T - \gamma]} (\la g_{m,n}, \clol_i^r(s+\gamma) \ra - \la g_{m,n}, \clol_i^r(s) \ra) \leq \gamma \eta_i \la g_{m,n}, \theta_i \ra + \dlt / (8N^2).
\end{align*}
We plug the last inequality and \eqref{eq20} into $\Sgm = \sum_{m,n \in \nat} \Sgm_{m,n}$, then \eqref{eq19} follows by \eqref{eq18} and the choice of~$\gamma$.
\end{proof}

The previous two lemmas are summed up into the following result.
\begin{lemma} \label{lem:as_reg_zr}
By \textup{(A.3)--(A.6)}, for any $T > 0$, $\dlt > 0$ and $\eps > 0$, there exists an $a > 0$ such that 
\[
\liminf_{r \to \infty} \pr^r \{ \sup_{t \in [0,T]} \sup_{x \in \nneg} \| \cloz^r(t)(H_x^{x+a}) \| \vee \| \cloz^r(t)(V_x^{x+a}) \| \leq \dlt \} \geq 1 - \eps.
\]
\end{lemma}

\begin{proof} 
Note that
\begin{align*}
&\sup_{x \in \nneg} \| \clozinit(t)(H_x^{x+a}) \| \vee \| \clozinit(t)(V_x^{x+a}) \| \\
\leq &\sup_{x \in \nneg} \| \cloz^r(0)(H_x^{x+a}) \| \vee \| \cloz^r(0)(V_x^{x+a}) \|.
\end{align*}
Indeed,
\[
\clozinit_i (t)(H_x^{x+a}) \leq \cloz_i^r(0)(H_{x+t}^{x+a+t}) \quad \text{and} \quad 
\clozinit_i(t) (V_x^{x+a}) \leq \cloz_i^r(0)(V_{x+S_i(\oz^r,0,t)}^{x+a+S_i(\oz^r,0,t)}).
\]
Then the lemma follows by $\cloz^r(\cdot) = (\clozinit + \cloznew)(\cdot)$ and Lemmas \ref{lem:as_reg_z0} and \ref{lem:as_reg_znew}.
\end{proof}

\subsection{Oscillation control} \label{sec:oscillation_control}
Here we establish the second key ingredient of tightness of the scaled state descriptors, the first one is proven in Section~\ref{sec:compact_containment}.
\begin{lemma}
By \textup{(A.3)--(A.6)}, for any $T > 0$, $\dlt > 0$ and $\eps > 0$, there exists an $h > 0$ such that
\[
\liminf_{r \to \infty} \pr^r \{ \underbrace{\omega(\cloz^r, h, T) \leq \dlt}_{\displaystyle{=: \Omega_\ast^r}} \} \geq 1 - \eps,
\]
where $\omega(\cloz^r, h, T) := \sup \{ d_I(\cloz^r(s), \cloz^r(t)) \colon s,t \in [0,T], \ |s - t| < h \}$.
\end{lemma}

\begin{proof}
Fix $T$, $\dlt$ and $\eps$. By \textup{(A.3)},
\[
\lim_{r \to \infty} \pr^r \{ \underbrace{\sup_{t \in [0,T]} \| \ove^r(t) - t \eta \| \leq \dlt / 4}_{\displaystyle{=:\Omega_1^r}} \} = 1.
\]
By Lemma~\ref{lem:as_reg_zr}, there exists an $a > 0$ such that
\[
\liminf_{r \to \infty} \pr^r \{ \underbrace{ \sup_{t \in [0,T]} \| \cloz^r(t)(H_0^a \cup V_0^a) \| \leq \dlt }_{\displaystyle{=:\Omega_2^r}} \} \geq 1 - \eps.
\]
Pick an $h$ such that $h (\|m\| \vee 1) \leq \dlt \wedge a$ and $h \| \eta \| \leq \dlt / 2$. We now show that, for all $r \in \clr$, $\Omega_\ast^r 
\supseteq \Omega_1^r \cap \Omega_2^r$, then the lemma follows.

Fix $r \in \clr$, $i$ and $s,t \in [0,T]$ such that $s < t$, $t - s < h$. Also fix an outcome $\omega \in \Omega_1^r \cap \Omega_2^r$. All random objects in the rest of the proof will be evaluated at this $\omega$. We have to check that, for any non-empty closed Borel subset $B \subseteq \nneg^2$,
\begin{subequations}
\begin{align}
\cloz_i^r(s)(B) &\leq \cloz_i^r(t)(B^\dlt) + \dlt, \label{eq:zs<zt} \\
\cloz_i^r(t)(B) &\leq \cloz_i^r(s)(B^\dlt) + \dlt. \label{eq:zt<zs}
\end{align}
\end{subequations}
First we check~\eqref{eq:zs<zt}. Note that it suffices to show
\begin{equation} \label{eq16}
\cloz_i^r(s)(B) \leq \cloz_i^r(\tau)(B^\dlt) + \dlt, 
\end{equation}
where $\tau := \inf \{ u \in [s,t] \colon \oz_i^r(u) = 0 \}$ and infimum over the empty set equals $t$ by definition. Indeed, if $\tau  = t$, then~\eqref{eq16} implies~\eqref{eq:zs<zt}. If $\tau < t$, then by the right-continuity of $\oz_i^r(\cdot)$, $\cloz_i^r(\tau)(B^\dlt) = \oz_i^r( \tau ) = 0$, and again~\eqref{eq16} implies~\eqref{eq:zs<zt}. 

Now prove~\eqref{eq16}. If $\tau = s$, then~\eqref{eq16} holds. Assume that $\tau > s$. By the defintion of $\Omega_2^r$,
\begin{equation} \label{eq21}
\cloz_i^r(s)(B) \leq \cloz_i^r(s)(B \cap [a,\infty)^2) + \dlt.
\end{equation}
Since $S_i(\oz^r,s,\tau) < \|m\| h \leq \dlt \wedge a$ and $\tau - s < h \leq \dlt \wedge a$,
\[
\cloz_i^r(s)(B \cap [a,\infty)^2) \leq \cloz_i^r(\tau)(B^\dlt),
\]
which together with~\eqref{eq21} implies~\eqref{eq16}.

It is left to check~\eqref{eq:zt<zs}. Since $S_i(\oz^r,s,\tau) < \|m\| h \leq \dlt$ and $\tau - s < h \leq \dlt$,
\[
\cloz_i^r(t)(B) \leq \cloz_i^r(s)(B^\dlt) + (\ove_i^r(t) - \ove_i^r(s)),
\]
and~\eqref{eq:zt<zs} follows by the definition of $\Omega_1^r$.
\end{proof}

\subsection{Fluid limits are bounded away from zero} \label{sec:fluid_limits_bounded_from_zero}
Rate constraints provide infinite-server-queue lower bounds for bandwidth-sharing networks. First we show that properly scaled infinite server queues are bounded away from zero, and then the same follows for bandwidth-sharing networks with rate constraints.

Consider a sequence of infinite server queues marked by $r \in \clr$. At $t = 0$, the queues are empty. To the r-th queue, customers arrive according to a counting process $A^r(\cdot)$ and have i.i.d. service times $\{ B_k^r \}_{k \in \nat}$ distributed as $B^r$. Let $\oa^r(\cdot) := A^r(\cdot)/r \Rightarrow \alpha(\cdot)$, where $\alpha(t) : = t \alpha$ and $\alpha > 0$. Also let $B^r \Rightarrow B$, where $\pr \{ B > 0 \} > 0$. Denote by $Q^r(\cdot)$ the population process of the $r$-th queue and put $\oq^r(\cdot) := Q^r(\cdot) / r$.

\begin{lemma} \label{lem:inf_server_queue}
For any $\dlt > 0$, there exists a $C(\dlt) > 0$ such that, for any $\Dlt > \dlt$,
\[
\pr^r \{ \inf_{\dlt \leq t \leq \Dlt} \oq^r(t) \geq C(\dlt) \} \to 1 \quad \text{as $r \to \infty$}.
\] 
\end{lemma}

{\it Proof.}
Let us first explain the result heuristically. Consider the arrivals with long service times, i.e. exceeding a $b > 0$. During $(0, b/2]$, there are $r \alpha \pr \{ B > b \} b /2$ such arrivals to the $r$-th queue. They will leave the queue after $t = b$, and hence, in $(b/2,b]$, the scaled queue length $\oq^r(\cdot)$ is bounded from below by $\alpha \pr \{ B > b \} b /2$. Similarly, $\oq^r(\cdot) \geq \alpha \pr \{ B > b \} b /2$ in any interval $((n-1)b/2, nb/2]$, $n \in \nat$.

We now proceed more formally. Pick an $b \in (0,\dlt)$ such that $b$ is a continuity point for the distribution of $B$, and
\[
p: = \pr \{ B \geq b \} > 0.
\]
Then, as $r \to \infty$,
\[
p_r := \pr^r \{ B^r \geq b \} \to p.
\]
Partition $(0,\Dlt]$ into subintervals of length $b/2$,
\[
(0,\Dlt] \subseteq \bigcup_{1 \leq n \leq N(\Dlt)} ((n-1)b/2, nb/2].
\]
Denote by $\oa^r_n$ the scaled number of arrivals during $((n-1)b/2, nb/2]$, and by $\oa^r_n(b)$ the scaled number of arrivals during $((n-1)b/2, nb/2]$ with service times at least $b$,
\begin{align*}
\oa^r_n &:= \oa^r(nb/2) - \oa^r((n-1)b/2), \\
\oa^r_n(b) &:= \frac{1}{r} \sum_{k = A^r((n-1)b/2)+1}^{A^r(nb/2)} \ind_{\{ B_k^r \geq b \}}.
\end{align*}
By $\oa^r \Rightarrow \alpha(\cdot)$ and $p_r \to p$ as $r \to \infty$, 
\begin{align*}
(\oa_1^r, \ldots, \oa_{N(\Dlt)}^r) &\Rightarrow (\alpha b/2, \ldots, \alpha b/2),\\
(\oa_1^r(b), \ldots, \oa_{N(\Dlt)}^r(b)) &\Rightarrow (\alpha p  b/2, \ldots, \alpha p b/2).
\end{align*}

Pick a $C(\dlt) < \alpha p b/2$, then
\begin{align*}
&\ \pr^r \{  \inf_{\dlt \leq t \leq \Dlt} \oq^r(t) \geq C(\dlt) \} \\
\geq& \ \pr^r \{ \inf_{t \in ((n-1)b/2, nb/2]} \oq^r(t) \geq C(\dlt), \ n = 2, \ldots, N(\Dlt)\} \\
 \geq& \ \pr^r \{ \oa^r_n(b) \geq C(\dlt), \ n = 1, \ldots, N(\Dlt) - 1\} \to 1 \quad \text{as $r \to \infty$}. \tag*{\qed}
\end{align*}

We can now prove easily that all fluid limits are bounded away from zero outside $t = 0$.
\begin{lemma} \label{lem:fluid_limit_pos}
For any $\dlt > 0$, there exists a $C(\dlt) > 0$ such that, for any fluid limit $(\clz,Z)(\cdot)$,
\[
\text{a.s.} \quad \inf_{t \geq \dlt} \min_{1 \leq i \leq I} Z_i(t) \geq C(\dlt).
\]
\end{lemma}

\begin{proof}
Consider a flow $k$ on route $i$ in the $r$-th network. By the rate constraints, this flow will stay in the network at least for $B_{ik}^r/m_i \wedge D_{ik}^r$ since its arrival. Hence, the route $i$ population process $Z_i^r(\cdot)$ is bounded from below by the length $Q_i^r(\cdot)$ of the infinite server queue with arrivals $E_i^r(\cdot)$ and i.i.d. service times $\{ B_{ik}^r / m_i \wedge D_{ik}^r \}_{k \in \nat}$. Assume that $Q_i^r(0) = 0$ and put $\oq_i^r(\cdot) = Q_i^r(\cdot) / r$. Then, by Lemma~\ref{lem:inf_server_queue}, for any $\dlt > 0$ there exists a $C(\dlt) > 0$ such that, for any $\Dlt > \dlt$,
\[
\pr^r \{ \inf_{t \in [\dlt, \Dlt]} \min_{1 \leq i \leq I} \oz_i^r(t) \geq C(\dlt) \} \geq \pr^r \{ \inf_{t \in [\dlt, \Dlt]} \min_{1 \leq i \leq I} \oq_i^r(t) \geq C(\dlt) \} \to 1.
\]
Now consider a fluid limit $(\clz,Z)(\cdot)$ along a subsequence $\{ (\cloz^q, \oz^q)(\cdot) \}_{q \in \clq}$. For any compact set $K \subset \nneg$, the mapping $\varphi_K : \skr_{\nneg \to \real} \to \real$, $\varphi_K(x) := \inf_{t \in K} \min_{q \leq i \leq I} x(t)$ is continuous at continuous $x(\cdot)$. Hence, $\varphi_{[\dlt, \Dlt]} (\oz^q) \Rightarrow \varphi_{[\dlt, \Dlt]} (Z)$ and, by the Portmanteau theorem,
\[
\pr \{ \varphi_{[\dlt, \Dlt]} (Z) \geq C(\dlt) \} \geq \limsup_{q \to \infty} \pr^q \{ \varphi_{[\dlt, \Dlt]} (\oz^q) \geq C(\dlt) \}  = 1,
\]
where $\Dlt > \dlt$ is arbitrary. Then the lemma follows.

Note also that the constant $C(\dlt)$ does not depend on a particular fluid limit $(\clz,Z)(\cdot)$.
\end{proof}

\subsection{Fluid limits as fluid model solutions} \label{sec:fluid_limits_fms} Here we show that fluid limits a.s. satisfy the fluid model equation~\eqref{eq:mvfms}.

Let $(\clz,Z)(\cdot)$ be a fluid limit along a subsequence $\{ (\cloz^q,\oz^q)(\cdot) \}_{q \in \clq}$. Lemma~\ref{lem:as_reg_zr} implies that ({\it cf.}~the~proof of~\cite[Lemma~6.2]{GRZ08})
\begin{equation} \label{eq25}
\text{a.s.} \quad \clz_i(t)(\partial_A) = 0 \quad \text{for all $t \geq 0$, all $i$ and $A \in \clc$},
\end{equation}
where $\partial_A$ denotes the boundary of $A$. Then, when proving~\eqref{eq:mvfms} for $(\clz,Z)(\cdot)$, it suffices to consider sets $A$ from
\[
\clc^+ := \{ [x,\infty) \times [y,\infty) \colon x \wedge y > 0 \}.
\]
It also suffices to consider $t$ from a finite interval $[0,T]$.

The rest of the proof splits into two parts. First we derive dynamic equations for the prelimiting processes $(\cloz^q,\oz^q)(\cdot)$, and then show that these equations converge to \eqref{eq:mvfms}.

\paragraph{Prelimiting equations} Fix $q \in \clq$, $i$, $t \leq T$ and $A \in \clc^+$. Fix also an outcome $\omega \in \Omega^q$. In what follows up to equation~\eqref{eq24}, all random elements are evaluated at this $\omega$. We have
\begin{equation} \label{eq22}
\begin{split}
\cloz_i^q(t)(A) =& \cloz_i^q(0)(A + (S_i(\oz^q,0,t),t)) \\
&+ \underbrace{\frac{1}{q} \sum_{k = 1}^{E_i^q(t)} \overbrace{\ind_A (B_{ik}^q - S_i(\oz^q,U_{ik}^q,t), 
D_{ik}^q - (t - U_{ik}^q))}^{\displaystyle{=: s_k}}}_{\displaystyle{ =: \Sgm}}.
\end{split}
\end{equation}
Fix a partition partition $0 < t_0 < t_1  < ... < t_N = t$, then
\[
\Sgm = \frac{1}{q}\sum_{k = 1}^{E_i^q(t_0)} s_k + \frac{1}{q}\sum_{j = 0}^{N-1} \sum_{k = E_i^q(t_j) + 1}^{E_i^q(t_{j+1})} s_k.
\]
Suppose that a function $y(\cdot)$ is non-increasing in $[t_0,t]$ and that, for some $\dlt$,
\[
\sup_{s \in [t_0,t]} | S_i(\oz^q, s, t) - y(s) | \leq \dlt.
\]
Now we can estimate $\Sgm$. If $U_{i k}^q \in (t_j, t_{j+1}]$, then
\begin{gather*}
B_{i k}^q - (y(t_j) + \delta) \leq B_{i k}^q - S(\oz^q, U_{i k}^q, t) \leq B_{i k}^q - (y(t_{j+1}) - \delta),\\
D_{i k}^q - (t - t_j)) \leq D_{i k}^q - (t - U_{i k}^q) \leq D_{i k}^q - (t - t_{j+1}),
\end{gather*}
and
\begin{align*}
\Sigma &\geq \sum_{j=0}^{N-1} \dfrac{1}{q} \sum_{k= E_i^q(t_j)+1}^{E_i^q(t_{j+1})} \ind_A (B_{i k}^q - (y(t_j) + \delta), D_{i k}^q - (t - t_j)), \\
\Sigma &\leq \ove_i^q(t_0) + \sum_{j=0}^{N-1} \dfrac{1}{q} 
\sum_{k= E_i^q(t_j)+1}^{E_i^q(t_{j+1})} \ind_A (B_{i k}^q - (y(t_{j+1}) - \delta), D_{i k}^q - (t - t_{j+1})),
\end{align*}
which can be rewritten as
\begin{equation} \label{eq23}
\begin{split}
\Sigma &\geq \sum_{j=0}^{N-1} \clol_i^q(t_j,t_{j+1})(A + (y(t_j) + \delta,t - t_j)) \\
\Sigma &\leq \ove_i^q(t_0) + \sum_{j=0}^{N-1} \clol_i^q(t_j,t_{j+1})(A+(y(t_{j+1}) - \delta,t - t_{j+1})).
\end{split}
\end{equation}
Put
\[
X^q := \sup_{A \in \clc} \sup_{0 \leq s \leq t \leq T} \| (\clol^q(s,t)(A) -  (t - s) \eta \ast \theta^q(A) \|,
\]
then, by~\eqref{eq23} and~\eqref{eq22},
\begin{equation} \label{eq24}
\begin{split}
&\sum_{j=0}^{N-1} \Bigl( \eta_i (t_{j+1} - t_j) \theta_i^q(A + (y(t_j) + \delta, t - t_j)) - X^q \Bigr)  \\
\leq \ &\cloz_i^q(t)(A) - \cloz_i^q(0)(A + (S_i(\oz^q,0,t),t))  \\
\leq \ &\eta_i t_0 + X^q + \sum_{j=0}^{N-1} \Bigl( \eta_i (t_{j+1} - t_j) \theta_i^q(A+(y(t_{j+1}) - \delta, t - t_{j+1}) + X^q \Bigr).
\end{split}
\end{equation}
To summarize, we have shown that, for all $q \in \clq$ and $\omega \in \Omega^q$,
\begin{equation} \label{eq26}
(\cloz^q(\cdot),X^q) \in \cla^q,
\end{equation}
where $\cla^q \subset \skr_{\nneg \to \msr^I} \times \nneg$ is the set of pairs $(\zeta(\cdot),x)$ such that, for any set $A \in \clc^+$, any partition $0 < t_0 < t_1 < \ldots < t_N = t \leq T$ and any function $y(\cdot)$ that is non-increasing in $[t_0,t]$ and that satisfies $\sup_{s \in [t_0,t]} |S_i(\la 1, \zeta \ra, s, t) - y(s)| \leq \dlt$ for some $i$ and $\dlt$,
\begin{equation*}
\begin{split}
&\sum_{j=0}^{N-1} \Bigl( \eta_i (t_{j+1} - t_j) \theta_i^q(A + (y(t_j) + \delta, t - t_j)) - x \Bigr)  \\
\leq \ &\zeta(t)(A) - \zeta_i(0)(A + (S_i(\la 1, \zeta \ra ,0,t),t))  \\
\leq \ &\eta_i t_0 + x + \sum_{j=0}^{N-1} \Bigl( \eta_i (t_{j+1} - t_j) \theta_i^q(A+(y(t_{j+1}) - \delta, t - t_{j+1}) + x \Bigr).
\end{split}
\end{equation*}

\paragraph{Limiting equations}
By \textup{(A.3) {\it and} (A.4)} ({\it cf.} the proof of~\cite[Lemma~5.1]{GRZ08}),
\[
X_q \Rightarrow 0 \quad \text{as } q \to \infty.
\]
Since the limit of $X_q$ is deterministic, then the joint convergence $(\cloz^q(\cdot), X^q) \Rightarrow (\clz(\cdot),0)$ holds. By the Skorokhod representation theorem, there exist random elements $\{ \tilde{\clz}^q(\cdot) \}_{q \in \clq}$, $\tilde{\clz}(\cdot)$ and $\{ \tilde{X}^q \}_{q \in \clq}$ defined on a common probability space $(\tilde{\Omega}, \tilde{\mathcal{F}}, \tilde{\pr})$ such that $(\tilde{\clz}^q(\cdot), \tilde{X}^q) \stackrel{\text{d}}{=} (\cloz^q(\cdot), X^q)$, $q \in \clq$, and $\tilde{\clz}(\cdot) \stackrel{\text{d}}{=} \clz(\cdot)$, and
\begin{equation} \label{eq27}
\text{a.s.} \quad (\tilde{\clz}^q(\cdot),\tilde{X}^q) \to (\tilde{\clz}(\cdot),0) \quad \text{as $q \to \infty$}.
\end{equation}
Introduce also the total mass processes $\tilde{Z}^q(\cdot) := \la 1, \tilde{\clz}^q(\cdot) \ra$, $q \in \clq$, and $\tilde{Z}(\cdot) := \la 1, \tilde{\clz}(\cdot) \ra$. By~Lemma~\ref{lem:fluid_limit_pos}, \eqref{eq25} and \eqref{eq26},
\begin{subequations} \label{eq:as_fluid_limit}
\begin{align}
\text{a.s.} \quad \tilde{Z}_i(t) > 0 \quad &\text{for all $t>0$ and all $i$},  \label{eq:as_fluid_limit_a} \\
\text{a.s.} \quad \tilde{\clz}_i(t)(\partial_A) = 0 \quad &\text{for all $t \geq 0$, all $i$ and $A \in \clc$}, \label{eq:as_fluid_limit_b}\\
\text{a.s.} \quad (\tilde{\clz}^q(\cdot),\tilde{X}^q) \in \cla^q \quad &\text{for all $q \in \clq$}. \label{eq:as_fluid_limit_c}
\end{align}
\end{subequations}
Denote by $\tilde{\Omega}_\ast$ the set of outcomes $w \in \tilde{\Omega}$ for which \eqref{eq27} and \eqref{eq:as_fluid_limit} hold. We will show that, for all $\omega \in \tilde{\Omega}_\ast$, all $i$, $t \in [0,T]$ and $A \in \clc^+$,
\begin{equation} \label{eq33}
\begin{split}
\tilde{\clz}_i(t) (A) =& \tilde{\clz}_i(0) (A + (S_i(\tilde{Z},0,t), t)) \\
&+ \eta_i \int_0^t \theta_i(A + (S_i(\tilde{Z},s,t), t - s)) ds,
\end{split}
\end{equation}
and that will complete the proof of Theorem~\ref{th:fluid_limits}.

Fix $t \in [0,T]$, $i$ and $A \in \clc^+$. Also fix an outcome $\omega \in \tilde{\Omega_\ast}$. All random elements in the rest of the proof are evaluated at this $\omega$.

By~\eqref{eq27} and  \eqref{eq:as_fluid_limit_b},
\begin{equation} \label{eq30}
\tilde{\clz}_i^q(t)(A) \to \tilde{\clz}_i(t)(A) \quad \text{as $q \to \infty$}.
\end{equation}
By~\eqref{eq:as_fluid_limit_a}, the rate constraints and the dominated convergence theorem,
\begin{equation} \label{eq28}
S_i(\tilde{Z}^q, s,t) \to S_i(\tilde{Z}, s,t) \quad \text{for all $s \in [0,t]$} \quad \text{as $q \to \infty$}, 
\end{equation}
which in particular implies that
\begin{equation} \label{eq31}
\tilde{\clz}_i^q(0)(A + (S_i(\tilde{Z}^q,0,t),t)) \to \tilde{\clz}_i(0)(A + (S_i(\tilde{Z},0,t),t))  \quad \text{as $q \to \infty$}.
\end{equation}
Fix $t_0 \in (0,t)$ and $\dlt>0$. By~\eqref{eq:as_fluid_limit_a}, the function $S_i(\tilde{Z},\cdot,t)$ is continuous in $[t_0,t]$, and the functions $S_i(\tilde{Z}^q,\cdot,t)$ are monotone. Then the point-wise convergence~\eqref{eq28} implies uniform convergence in~$[t_0,t]$, and for $q$ large enough,
\begin{equation} \label{eq34}
\sup_{s \in [t_0,t]} |S_i(\tilde{Z}^q, s,t) - S_i(\tilde{Z}, s,t)| \leq \dlt.
\end{equation}
Now fix a partition $t_0<t_1<\ldots<t_N = t$. The bound~\eqref{eq34} and~\eqref{eq:as_fluid_limit_c} imply that (in the definition of $\cla^q$ we take $y(\cdot) = S_i(\tilde{Z}, \cdot,t)$)
\begin{equation}\label{eq29}
\begin{split}
&\sum_{j=0}^{N-1} \Bigl( \eta_i (t_{j+1} - t_j) \theta_i^q(A + (S_i(\tilde{Z}, t_j,t) + \delta, t - t_j)) - \tilde{X}^q \Bigr) \\
\leq \ &\tilde{\clz}_i^q(t)(A) - \tilde{\clz}_i^q(A + (S_i(\tilde{Z}^q, 0,t),t))  \\
\leq \ &\eta_i t_0 + \tilde{X}^q + \sum_{j=0}^{N-1} \Bigl( \eta_i (t_{j+1} - t_j) \theta_i^q(A+(S_i(\tilde{Z}, t_{j+1},t) - \delta, t - t_{j+1}) + \tilde{X}^q \Bigr). 
\end{split}
\end{equation}
Since $\theta_i(\cdot \times \nneg)$ and $\theta_i(\nneg \times \cdot)$ are probability measures, the set of $B \in \clc$ for which $\theta_i( \partial_B) > 0$ is at most countable. By~\eqref{eq:as_fluid_limit}, $S_i(\tilde{Z}, \cdot,t)$ is strictly monotone in $[t_0,t]$. Hence, the set $\cld$ of $s \in [t_0,t]$ for which
$\theta_i(\partial_{A + (S_i(\tilde{Z}, s,t) + \delta, t - s)})> 0$ or $\theta_i(\partial_{A+(S_i(\tilde{Z}, s,t) - \delta, t - s)}) > 0$ is at most countable, too. In~\eqref{eq29}, let  $q \to \infty$ assuming that the partition contains no points from $\mathcal{D}$. Then, by~\eqref{eq27}, \eqref{eq30} and~\eqref{eq31},
\begin{equation} \label{eq32}
\begin{split}
&\sum_{j=0}^{N-1} \eta_i (t_{j+1} - t_j) \theta_i(A + (S_i(\tilde{Z}, t_j,t) + \delta, t - t_j)) \\
\leq \ &\tilde{\clz}_i(t)(A) - \tilde{\clz}_i(0)(A + (S_i(\tilde{Z}, 0,t),t))  \\
\leq \ &\eta_i t_0 + \sum_{j=0}^{N-1} \eta_i (t_{j+1} - t_j) \theta_i(A+(S_i(\tilde{Z}, t_{j+1},t) - \delta, t - t_{j+1}). 
\end{split}
\end{equation}
Now, in~\eqref{eq32}, let the diameter of the partition go to $0$ keeping $t_0$ fixed. Then
\begin{equation*}
\begin{split}
& \eta_i \int_{t_0}^t \theta_i(A + (S_i(\tilde{Z}, s,t) + \delta, t - s)) ds \\
\leq \ &\tilde{\clz}_i(t)(A) - \tilde{\clz}_i(0)(A + (S_i(\tilde{Z}, 0,t),t))  \\
\leq \ &\eta_i t_0 + \eta_i \int_{t_0}^t \theta_i(A+(S_i(\tilde{Z}, s,t) - \delta, t - s) ds. 
\end{split}
\end{equation*}
Finally, in the last inequality, let $\dlt \to 0$ (recall~\eqref{eq:as_fluid_limit_b}) and $t_0 \to 0$, then \eqref{eq33} follows.

\section{Proof of Theorem~\ref{th:stationary_distributions}} \label{sec:proof_stat_distributions}

By the discussion following Theorem~\ref{th:stationary_distributions} and Lemma~\ref{lem:stat_no_atoms}, it is left to show tightness of the scaled stationary distributions. It suffices to show coordinate-wise tightness, so fix $i$. By~\cite[Theorem~2.1]{Jakubowski} and \cite[Theorem~15.7.5]{Kallenberg}, the  sequence $\{ \cloy_i^r, \oy_i^r \}_{r \in \clr}$ is tight if
\begin{subequations}
\begin{gather}
\sup_{r \in \clr} \ex^r \oy_i^r < \infty, \label{eq:stat_tight_a}\\
\lim_{n \to \infty} \ex^r \cloy_i^r(V_n^\infty) = 0, \label{eq:stat_tight_b}\\
\lim_{n \to \infty} \ex^r \cloy_i^r(H_n^\infty) = 0, \label{eq:stat_tight_c}
\end{gather}
\end{subequations}
where $V_n^\infty = [n,\infty) \times \nneg$ and $H_n^\infty = \nneg \times [n,\infty)$.

First check \eqref{eq:stat_tight_a}. For each $r$, the route~$i$ population process $Z_i^r(\cdot)$ is bounded from above by the length $Q_i^r(\cdot)$ of the $M/G/\infty$ queue with the following parameters:
\begin{itemize}
\item[(Q.1)] $\phantom{t}$ at $t = 0$, there are $Z_i^r(0)$ customers whose service times are patience times of the initial $\phantom{tt}$ flows on  route~$i$ of the $r$-th network;
\item[(Q.2)] $\phantom{t}$ the input process is the route~$i$ input process of the $r$-th network;
\item[(Q.3)] $\phantom{t}$ service times of newly arriving customers are patience times of newly arriving flows $\phantom{tt}$~on~route $i$ of the $r$-th network.
\end{itemize}
\vspace{2pt}
For all $r$ and $t$, $Z_i^r(t) \leq Q_i^r(t)$. As $t \to \infty$, $Z_i^r(t) \Rightarrow Y_i^r$ and $Q_i^r(t) \Rightarrow \Pi(\eta_i^r \ex^r D_i^r)$. Hence, $Y_i^r \leq_{\text{st}} \Pi(\eta_i^r \ex^r D_i^r)$ and $\ex^r \oy_i^r \leq \eta_i^r \ex^r D_i^r / r \to \eta_i \ex D_i$ as $r \to \infty$, which implies \eqref{eq:stat_tight_a}.

Now check~\eqref{eq:stat_tight_b}. Note that, if at some point the residual flow size is at least $n$, then the initial flow size was at least $n$, too. Hence, $\clz_i^r(\cdot)(V_n^\infty)$ is bounded from above by the length $Q_i^{r,n}(\cdot)$ of the $M/G/\infty$ queue whose initial state is as in (Q.1), newly arriving customers are newly arriving flows on route~$i$ of the $r$-th network with initial sizes at least $n$, and service times of newly arriving customers are patience times of the corresponding flows. In particular, the input process for this queue is Poisson with intensity $\eta_i^r \pr^r\{ B_i^r \geq n \}$. 

Let $f_n(\cdot)$ be a continuous function on $\nneg^2$ such that
\[
\ind_{V_{n+1}^\infty}(\cdot) \leq f_n(\cdot) \leq \ind_{V_n^\infty}(\cdot).
\]
Then, for all $r$ and $t$,
\[
\la f_n, \clz_i^r(t) \ra \leq \clz_i^r(t)(V_n^\infty) \leq Q_i^{r,n}(t)
\]
Letting $t \to \infty$, we obtain
\begin{gather*}
\cly_i^r(V_{n+1}^\infty) \leq \la f_n, \cly_i^r \ra \leq_{\text{st}} \Pi(\eta_i^r \pr^r\{ B_i^r \geq n \} \ex^r D_i^r), \\
\ex^r \cloy_i^r(V_{n+1}^\infty) \leq \eta_i^r \pr^r\{ B_i^r \geq n \} \ex^r D_i^r/ r,
\end{gather*}
and then~\eqref{eq:stat_tight_b} follows.

Finally,~\eqref{eq:stat_tight_c} is valid due to the following lemma.

\begin{lemma} \label{lem2}
For any $r \in \clr$, $i$ and Borel set $S \subseteq \nneg$,
\[
\cly_i^r(\nneg \times S) \leq_{\text{st}} \Pi(\eta_i^r \, \ex^r \! D_i^r \, \pr^r \! \{ \tilde{D}_i^r \in S \}),
\]
where $\tilde{D}_i^r$ has density $\pr^r\{ D_i^r > x \} / \ex^r D_i^r$, $x \geq 0$.
\end{lemma}

{\it Proof.}
Fix $r \in \clr$, $i$ and a Borel set $S \subseteq \nneg$. It suffices to show that, for any $\dlt>0$,
\[
\cly_i^r(\nneg \times S) \leq_{\text{st}} \Pi(\eta_i^r \ex^r D_i^r \pr^r \{ \tilde{D}_i^r \in S^\dlt \}),
\]
so fix $\dlt > 0$. 

Consider the upper bound queue $Q_i^r(\cdot)$ with parameters (Q.1)--(Q.3). Denote by $Q_i^r(t)(S^\dlt)$ the~number of customers in this queue whose residual service times at time $t$ are in $S^\dlt$. Then
\[
\clz_i^r(\cdot)(\nneg \times S^\dlt) \leq Q_i^r(\cdot)(S^\dlt).
\]
Given at time $t$ there are $k$ customers in the queue, denote by $D_1(t), \ldots, D_k(t)$ their residual service times. By~\cite[Chapter~3.2, Theorem~2]{Takacs}, 
\[
\lim_{t \to \infty} \pr^r \{ D_1(t) \leq x_1, \ldots, D_k(t) \leq x_k | Q_i^r(t) = k \} = \pr^r\{ \tilde{D}_i^r \leq x_1 \} \ldots \pr^r\{ \tilde{D}_i^r \leq x_k \},
\]
which together with $Q_i^r(t) \Rightarrow \Pi(\eta_i^r \ex^r D_i^r)$ as $t \to \infty$ implies that
\[
Q_i^r(t)(S^\dlt) \Rightarrow \Pi(\eta_i^r \ex^r D_i^r \pr^r \{ \tilde{D}_i^r \in S^\dlt \}).
\]
Let $g_\dlt$ be a continuous function on $\nneg^2$ such that 
\[
\ind_{\nneg \times S}(\cdot) \leq g_\dlt(\cdot) \leq \ind_{\nneg \times S^\dlt}(\cdot).
\]
Then, for any $t$,
\[ 
\la g^\dlt,  \clz_i^r(t) \ra \leq \clz_i^r(t)(\nneg \times S^\delta) \leq Q_i^r(t)(S^\delta),
\]
and as $t \to \infty$,
\begin{align*}
\cly_i^r(\real_+ \times S) \leq \la g^\dlt,  \cly_i^r \ra \leq_{\text{st}} \Pi (\eta_i^r \ex^r D_i^r \pr^r \{ \tilde{D}_i^r \in S^\dlt \}). \tag*{\qed}
\end{align*}

\section*{Appendix}
\appendix

{\it Proof of Lemma~\ref{lem:Lmb_continuous}.}
It suffices to show that, for a~vector $z \in \nneg^I$ with the first $I' < I$ coordinates positive and the rest of them zero, and a~sequence $\{ z^k\}_{k \in \nat} \subset \pos^I$ such that $z^k \to z$, we have $\Lmb(z^k) \to \Lmb(z)$.

Suppose that $z^k \to z$ but $\Lmb(z^k) \not \to \Lmb(z)$. Since $\{ \Lmb(z^k) \}_{k \in \nat}$ is a~subset of the compact set $\{\Lmb \in \nneg^I \colon \|\Lmb\| \leq \|C\| \}$, without loss of generality we may assume that $\Lmb(z^k) \to \Lmb' \not = \Lmb(z)$.

Recall that $\Lmb(z)$ is the unique optimal solution to
\begin{equation} \label{eq:def_Lmb}
\text{maximize} \quad \sum_{i=1}^I z_i \, \mathcal{U}_i( \Lmb_i / z_i ) \quad \text{subject to} \quad A\Lmb \leq C, \quad \Lmb \leq m \ast z,
\end{equation}
where, by convention, $\Lmb_i / 0 := 0$ and $0 \times (-\infty) := 0$.

For all $k$, $A \Lmb(z^k) \leq C$ and $\Lmb(z^k) \leq m \cdot z^k$. Hence, $\Lmb'$ is feasible for~\eqref{eq:def_Lmb} and $\Lmb'_i = 0 = \Lmb_i(z)$ for $i > I'$. Since $\Lmb' \neq \Lmb(z)$ is not optimal for~\eqref{eq:def_Lmb},
\begin{equation}\label{eq:contradiction}
l:=\sum_{i =1 }^{I'}  z_i \, \mathcal{U}_i( \Lmb_i (z) / z_i ) > \sum_{i = 1}^{I'} z_i \, \mathcal{U}_i( \Lmb'_i / z_i ) =: r.
\end{equation}
Now we construct a~sequence $\Lmb^k \to \Lmb(z)$ such that $\Lmb^k$  is feasible for the optimization problem~\eqref{eq:def_Lmb} with $z^k$ in place of $z$. Introduce vectors $C^k \in \nneg^J$ with $C_j^k = \sum_{i = I'+1}^I A_{j i} \Lmb_i(z^k)$. Put the first $I'$ coordinates of $\Lmb^k$ to be $\Lmb_i^k =  ( \Lmb_i (z) - \| C^k \| )^+ \wedge m_i z_i^k$, and the rest of them $\Lmb_i^k = \Lmb_i(z^k)$. That is, in the bandwidth allocation $\Lmb(z)$, the bandwidth $C^k$, which is required for the last $I - I'$ routes, is taken away from the first $I'$ routes.

Since $z^k \to z$, $\Lmb^k \to \Lmb(z)$ and $\Lmb(z^k) \to \Lmb'$,
\[
\sum_{i  = 1}^{I'} z_i^k \, \mathcal{U}_i( \Lmb_i^k / z_i^k ) \to l \quad \text{and} \quad \sum_{i = 1}^{I'} z_i^k \, \mathcal{U}_i( \Lmb_i (z^k) / z_i^k ) \to r.
\]
Also, for all $k$,
\[
\sum_{i =I'+1}^I z_i^k \, \mathcal{U}_i( \Lmb_i^k / z_i^k ) = \sum_{i = I'+1}^I z_i^k \, \mathcal{U}_i( \Lmb_i (z^k) / z_i^k ).
\]
Then, by~\eqref{eq:contradiction}, for $k$ big enough, 
\[
\sum_{i =1}^I z_i^k \, \mathcal{U}_i( \Lmb_i^k / z_i^k ) > \sum_{i = 1}^I z_i^k \, \mathcal{U}_i( \Lmb_i (z^k) / z_i^k ),
\]
which contradicts to $\Lmb(z^k)$ being optimal for~\eqref{eq:def_Lmb} with $z^k$ in place of $z$. \qed

{\it Proof of Corollary~\ref{cor:stable_fixed_point}.}
Fix an FMS $(\zeta,z)(\cdot)$. In Section~\ref{sec:fluid_model}, we discussed how Theorem~\ref{th:stable_fixed_point} implies that $z(t) \to z^\ast$ as $t \to \infty$. Here we prove that $z(t) \to z^\ast$ implies $\zeta(t) \to \zeta^\ast$. It suffices to show that, for any $\eps > 0$, there exists a $t_\eps$ such that, for all $t \geq t_\eps$, $i$ and Borel sets $A \subseteq \nneg^2$,
\begin{equation} \label{eq40}
\begin{split}
\zeta_i(t)(A) &\leq \zeta_i^\ast(A^\eps) + \eps, \\
\zeta_i^\ast(A) &\leq \zeta_i(t)(A^\eps) + \eps,
\end{split}
\end{equation}
so fix $\eps > 0$.

For any $\dlt \in (0, \min_{1 \leq i \leq I} z_i^\ast)$, there exists a $\tau_\dlt$ such that, for all $t \geq \tau_\dlt$,
\[
z^\ast - \dlt := (z_1^\ast - \dlt, \ldots, z_I^\ast - \dlt) \leq z(t) \leq (z_1^\ast + \dlt, \ldots, z_I^\ast + \dlt) =: z^\ast + \dlt.
\]
Then, for all $t \geq s \geq \tau_\dlt$ and $i$, we have
\[
\underbrace{r_i(z^\ast - \dlt, z^\ast + \dlt)}_{\displaystyle{=: r_i^\dlt}}(t - s) \leq S_i(z,s,t) \leq \underbrace{R_i(z^\ast - \dlt, z^\ast + \dlt)}_{\displaystyle{=: R_i^\dlt}}(t - s),
\]
which, when plugged into the shifted fluid model equation~\eqref{eq:mvfms_shifted}, implies that, for all~$t \geq \tau_\dlt$, $i$ and Borel sets $A \subseteq \nneg^2$,
\begin{subequations} \label{eq41}
\begin{align}
\zeta_i(t)(A) &\leq \zeta_i(\tau_\dlt)(\nneg^2 \times [t - \tau_\dlt, \infty)) +\eta_i \int_0^{t - \tau_\dlt} \theta_i(A+(r_i^\dlt s, s)) ds, \label{eq41a}\\
\zeta_i(t)(A) &\geq \eta_i \int_0^{t - \tau_\dlt} \theta_i(A+(R_i^\dlt s, s)) ds. \label{eq41b}
\end{align}
\end{subequations}
Recall from Section~\ref{sec:fluid_model} that, for all $i$ and Borel sets $A \subseteq \nneg^2$,
\begin{equation} \label{eq42}
\zeta_i^\ast(A) = \eta_i \int_0^\infty \theta_i(A + (\lmb_i(z^\ast)s,s)) ds.
\end{equation}
Now, there exists a $t'_\eps$ such that, for all~$i$, Borel sets $A \subseteq \nneg^2$ and $\dlt \in (0, \min_{1 \leq i \leq I} z_i^\ast)$,
\begin{subequations}
\begin{align}
\eta_i \int_{t'_\eps}^\infty \theta_i(A+(r_i^\dlt s, s)) ds \leq \eta_i \int_{t'_\eps}^\infty \pr \{ D_i \geq s \} ds \leq \eps/2, \label{eq43a} \\
\eta_i \int_{t'_\eps}^\infty \theta_i(A+(\lmb_i(z^\ast) s, s)) ds \leq \eta_i \int_{t'_\eps}^\infty \pr \{ D_i \geq s \} ds \leq \eps/2. \label{eq43b}
\end{align}
\end{subequations}
Take $\dlt \in (0, \min_{1 \leq i \leq I} z_i^\ast)$ such that
\[
\| R^\dlt - \lmb(z^\ast) \| t'_\eps \leq \eps /2  \quad \text{and} \quad \| r^\dlt - \lmb(z^\ast) \| t'_\eps \leq \eps /2.
\]
Then, for all $i$ and Borel sets $A \subseteq \nneg^2$,
\begin{subequations}
\begin{align}
\eta_i \int_0^{t'_\eps} \theta_i(A+(r_i^\dlt s, s)) ds &\leq \eta_i \int_0^{t'_\eps} \theta_i(A^\eps+(\lmb_i(z^\ast) s, s)) ds \label{eq44a} \\
\eta_i \int_0^{t'_\eps} \theta_i(A+(\lmb_i(z^\ast) s, s)) ds &\leq \eta_i \int_0^{t'_\eps} \theta_i(A^\eps+(R_i^\dlt s, s)) ds. \label{eq44b}
\end{align}
\end{subequations}
Also take $t''_\eps$ such that, for all $i$,
\begin{equation} \label{eq45}
\zeta_i(\tau_\dlt)(\nneg^2 \times [t''_\eps - \tau_\dlt, \infty) ) \leq \eps/2.
\end{equation}
Now we put~\eqref{eq41}--\eqref{eq45} together in order to obtain~\eqref{eq40}: for all $t \geq t_\eps := (\tau_\dlt + t'_\eps) \vee t''_\eps$, $i$ and Borel sets $A \subseteq \nneg^2$,
\begin{gather*}
\zeta_i(t)(A) \stackrel{\eqref{eq41a},\eqref{eq45}}{\leq} \eps/2 + \eta_i \int_0^{t - \tau_\dlt} \theta_i(A + (r_i^\dlt s, s)) ds 
\stackrel{\eqref{eq43a}}{\leq} \eps/2 + \eta_i \int_0^{t'_\eps} \theta_i(A + (r_i^\dlt s, s)) ds + \eps/2 \\
\stackrel{\eqref{eq44a}}{\leq} \eta_i \int_0^{t'_\eps} \theta_i(A^\eps + (\lmb_i(z^\ast) s, s)) ds + \eps \stackrel{\eqref{eq42}}{\leq} \zeta_i^\ast(A^\eps) + \eps
\end{gather*}
and
\begin{gather*}
\zeta_i^\ast(A) \stackrel{\eqref{eq42},\eqref{eq43a}}{\leq} \eta_i \int_0^{t'_\eps} \theta_i(A + (\lmb_i(z^\ast) s, s)) ds + \eps/2 \\
\stackrel{\eqref{eq44b}}{\leq} \eta_i \int_0^{t'_\eps} \theta_i(A^\eps + (R_i^\dlt s, s)) ds + \eps / 2 \stackrel{\eqref{eq41b}}{\leq} \zeta_i(t)(A^\eps) + \eps. \tag*{\qed}
\end{gather*}

{\it Proof of Lemma~\ref{lem1}.}
For all $s \leq t$ and $\eps > 0$,
\begin{align*}
\int_s^t \pr \{ u+x \leq \xi < u+x'+\eps \} du &= \int_{s + x}^{t + x} \pr \{ \xi \geq u \} du - \int_{s + x' + \eps}^{t + x' + \eps} \pr \{ \xi \geq u \} du \\
&\leq \int_{s + x}^{s + x' + \eps} \pr \{ \xi \geq u \} du \leq x' - x + \eps.
\end{align*}
The lemma follows as we first let $\eps \to 0$ (applying the dominated convergence theorem) and then $s \to -\infty$, $t \to \infty$. \qed


\end{document}